\documentclass[12pt]{amsart}
\usepackage{a4wide,enumerate,xcolor}
\usepackage{amsmath,comment}
\usepackage{scrextend} 
\usepackage{tikz}
\usepackage{pifont}
\allowdisplaybreaks

\let\pa\partial
\let\na\nabla
\let\eps\varepsilon
\newcommand{\N}{{\mathbb N}}
\newcommand{\R}{{\mathbb R}}
\newcommand{\Z}{{\mathbb Z}}
\newcommand{\diver}{\operatorname{div}}
\newcommand{\dd}{{\mathrm d}}
\newcommand{\Pb}{\mathbb{P}}
\newcommand{\E}{\mathbb{E}}

\newcommand{\K}{{\mathcal K}}
\renewcommand{\L}{{\mathcal L}}

\newtheorem{theorem}{Theorem}
\newtheorem{lemma}[theorem]{Lemma}
\newtheorem{proposition}[theorem]{Proposition}
\newtheorem{remark}[theorem]{Remark}
\newtheorem{corollary}[theorem]{Corollary}
\newtheorem{definition}{Definition}

\begin{document}

\title[Discretization schemes for Fokker--Planck equations]{Long-time behavior for discretization schemes of Fokker--Planck equations via couplings}

\author[A. J\"ungel]{Ansgar J\"ungel}
\address{Institute of Analysis and Scientific Computing, Vienna University of Technology, Wiedner Hauptstra\ss e 8--10, 1040 Wien, Austria}
\email{juengel@tuwien.ac.at}

\author[K. Schuh]{Katharina Schuh}
\address{Institute of Analysis and Scientific Computing, Vienna University of Technology, Wiedner Hauptstra\ss e 8--10, 1040 Wien, Austria}
\email{katharina.schuh@tuwien.ac.at} 

\date{\today}

\thanks{The authors acknowledge partial support from   
the Austrian Science Fund (FWF), grant 10.55776/F65, and from the Austrian Federal Ministry for Women, Science and Research and implemented by \"OAD, project MultHeFlo. This work has received funding from the European Research Council (ERC) under the European Union's Horizon 2020 research and innovation programme, ERC Advanced Grant NEUROMORPH, no.~101018153. For open-access purposes, the authors have applied a CC BY public copyright license to any author-accepted manuscript version arising from this submission.} 

\begin{abstract}
Continuous-time Markov chains associated to finite-volume discretization schemes of Fokker--Planck equations are constructed. Sufficient conditions under which quantitative exponential decay in the $\phi$-entropy and Wasserstein distance are established, implying modified logarithmic Sobolev, Poincar\'e, and discrete Beckner inequalities. The results are not restricted to additive potentials and do not make use of discrete Bochner-type identities. The proof for the $\phi$-decay relies on a coupling technique due to Conforti, while the proof for the Wasserstein distance uses the path coupling method. Furthermore, exponential equilibration for discrete-time Markov chains is proved, based on an abstract discrete Bakry--Emery method and a path coupling.
\end{abstract}

 

\maketitle


\section{Introduction}

We investigate finite-volume discretizations for Fokker--Planck equations, which preserve the exponential decay of the continuous equation with respect to the entropy and Wasserstein distance. Exponential equilibration is often proved by means of the Bakry--Emery method \cite{BaEm85}. The computations needed to apply this method use the chain rule, which is not easily available on the discrete level. Possible approaches, which avoid any discrete chain rule, are discrete Bochner-type inequalities \cite{CPP09} or nonlinear summation-by-parts formulas \cite{JuSc17}. Recently, Conforti suggested in \cite{Con22} a new approach that combines the Bakry--Emery method and coupling arguments as a probabilistic alternative to the discrete Bochner identities. In this paper, we extend his approach to Markov chains arising from finite-volume schemes for Fokker--Planck equations. An advantage of this approach is that we can allow for non-additive potentials, which were needed in, e.g., \cite{MaMa16}.

\subsection{The setting}

We are interested in the long-time behavior of finite-volume discretizations on rectangular grids of the Fokker--Planck equation
\begin{equation}\label{1.FP}
  \pa_t u = \sigma^2\Delta u + \diver(u\na V)
\end{equation}
on the cube $D=[-K,K]^d$ ($K>0$) with homogeneous Neumann boundary conditions and the initial condition $u(0,x)=u_0(x)$ for $x\in D$. The diffusion constant $\sigma>0$ is positive and the potential $V=V(x)$ is assumed to be smooth and strongly convex (see below). Equation \eqref{1.FP} is the analytic counterpart of the stochastic differential equation (SDE) 
\begin{equation}\label{1.sde}
  \dd X_t = -\na V(X_t)\dd t + \sqrt{2\sigma^2}\dd B_t
  + \sum_{j=1}^d e_j(\dd L^{j+}-\dd L^{j-})\quad\mbox{on }D
\end{equation} 
with the initial condition $X_0$ with law $u_0$, $e_j$ are the Euclidean basis vectors, and the local time $L^{j\pm}$ is a continuous, nondecreasing process such that $L_0^{j\pm}=0$, which increases only when the process $(X_t)_{t\ge 0}$ is in the half-spaces associated to $D$, i.e.\ $L_t^{j\pm}=\int_0^t\mathrm{1}_{\{X_s\in\pa D_j^\pm\}}\dd L_s^{j\pm}$, where $\pa D_j^\pm=\{x\in D:\langle x\mp K e_j,e_j\rangle = 0\}$ and $\langle\cdot,\cdot\rangle$ is the standard inner product on $\R^d$. The existence and uniqueness of a strong solution to \eqref{1.sde} holds by \cite{Ta79}. The terms involving $L^{j\pm}_t$ ensure that the solution to the SDE is reflected at the boundary $\pa D$. 

The probability density of the solution to the SDE \eqref{1.sde} solves the Fokker--Planck equation \eqref{1.FP}, and the corresponding generator to \eqref{1.sde} is of the form
\begin{equation}\label{1.L}
  \L f(x) = \sigma^2\Delta f(x) - \na V(x)\cdot\na f(x)\quad\mbox{for }
  f\in\mathcal{A},\ x\in D,
\end{equation}
where the domain of $\L$ is
\begin{equation}\label{1.A}
  \mathcal{A} = \{f\in C^2_0(\R^d;\R):\pa_j f(x)=0\quad\mbox{for }
  x\in\pa D_j^\pm,\ j=1,\ldots,d\}.
\end{equation}
Note that its adjoint $\L^*$ is given by the right-hand side of \eqref{1.FP}.  

The finite-volume discretization for \eqref{1.FP} is defined as follows. Let $\K_h$ be the centers of the cells of a regular rectangular grid (see Section \ref{sec.def} for details). We study the semi-discretized equation of \eqref{1.FP},
$$
  \pa_t u(t,i) = \L^*_h u(t,i)\quad\mbox{for }t>0,\ i\in\K_h,
$$
where $\L_h^*$ is a discretization of $\L^*$. As in the continuous case, this operator is the adjoint of some operator $\L_h$ of some stochastic process, namely a continuous-time Markov chain on the state space $\K_h$, and it can be written as
\begin{equation*}
  \L_h f(i) = \sum_{\gamma\in G_i}c(i,\gamma)\big(f(\gamma i)-f(i)\big),
\end{equation*}
where $G_i$ is the set of all moves $\pm_j$ to the neighboring states $i\mapsto i\pm he_j$ (with the grid size $h>0$) and $c(i,\gamma)$ denotes the rate that a jump happens from $i$ in the direction $\gamma$,
\begin{equation}\label{1.c}
  c(i,\pm_j) = \frac{\sigma^2}{h^2}\begin{cases}
  e^{-(V^h(i\pm he_j)-V^h(i))/(2\sigma^2)}
  & \mbox{if }i\pm he_j\in\K_h, \\
  0 & \mbox{else},
  \end{cases}
\end{equation}
where $V^h$ is the mean value of $V$ over the control volume with center $i$,
$$
  V^h(i) = \frac{1}{h^d}\int_{[-h/2,h/2]^d}V(i+s)\dd s.
$$
Because of the homogeneous Neumann boundary condition, the rate function for states at the boundary vanishes for moves going outside of the domain. We discuss the choice \eqref{1.c} in Remark \ref{rem.c}. The Markov chain associated to this scheme has a unique invariant measure $m_h$ on $\K_h$; see Lemma \ref{lem.invmeas}. 

\subsection{Main results and key ideas}

We suppose that
\begin{itemize}
\item[(i)] $\phi$ and $\Phi$, defined by $\Phi(a,b)=(\phi'(a)-\phi'(b))(a-b)$ for $a,b\in\R$, are convex;
\item[(ii)] the potential $V$ is strongly convex in the sense $\langle x-y,\na V(x)-\na V(y)\rangle\ge \kappa|x-y|^2$ for all $x,y\in\R^d$ and some $\kappa>0$;
\item[(iii)] the Hessian of the potential is strictly diagonally dominant in a sense specified in \eqref{3.kappa+}--\eqref{3.kappa-} below. 
\end{itemize}
Condition (i) is satisfied for the power-law entropies $\phi_\alpha(x)=x^\alpha$ with $1<\alpha\le 2$ and the Boltzmann entropy density $\phi_\alpha(x)=x\log x$ with $\alpha=1$. For this family of functions, the exponential decay is closely related to convex Sobolev inequalities, namely the modified logarithmic Sobolev inequality for $\alpha=1$, the Poincar\'e inequality for $\alpha=2$, and Beckner-type inequalities for $1<\alpha<2$. The convexity of $\phi$ is natural in this context, while the convexity of $\Phi$ is needed in Conforti's coupling approach; see Theorem \ref{thm.conforti} below. The strong convexity of the potential in condition (ii) is required to conclude exponential decay. Condition (iii) is satisfied, for instance, for additive potentials being of the form $V(x)=\sum_{j=1}^d V_j(x_j)$ with $x=(x_1,\ldots,x_d)\in\R^d$; see Remark \ref{rem.add}. An example of a non-additive potential satisfying condition (iii) is given by $V(x)=x^TKx$, where $K$ is a symmetric positive definite matrix with smallest eigenvalue $\kappa>0$.

Our main results include the following quantitative decay rates for the continuous-time Markov chain and its discrete-time version (obtained from an explicit Euler scheme) in the relative $\phi$-entropy
\begin{equation*}
  \mathcal{H}^\phi(f|m_h) = \sum_{i\in\K_h}\phi(f(i))m_h(i)
  - \phi\bigg(\sum_{i\in\K_h}f(i)m_h(i)\bigg)
\end{equation*}
and the $L^1$ Wasserstein distance $\mathcal{W}_1$:
\begin{align}
  \mathcal{H}^\phi(S_t f|m_h) 
  \le e^{-\kappa_\phi t}\mathcal{H}^\phi(f|m_h) &\quad\mbox{for }f\ge 0,
  \ t\ge 0, \label{1.res1} \\
  \mathcal{W}_1(\nu p_t,\eta p_t) 
  \le \sqrt{d}e^{-\kappa_1 t}\mathcal{W}_1(\nu,\eta) 
  &\quad\mbox{for }t\ge 0, \label{1.res2}
\end{align}
where $S_t$ denotes the semigroup of the Markov chain, i.e.\ $\pa_t(S_tf)=\L_h(S_tf)$, $\nu$ and $\eta$ are two probability measures, and $(p_t)_{t\ge 0}$ is the transition function of the Markov chain. The constants $\kappa_\phi>0$ and $\kappa_1>0$ depend on the gap of the strictly diagonally dominance for the Hessian of $V$ from condition (iii). 

The first idea of the proof of inequality \eqref{1.res1} is based on the Bakry--Emery method. Indeed, if the convex Sobolev inequality
\begin{equation}\label{1.csi}
  \kappa_\phi\mathcal{H}^\phi(f|m_h) \le \mathcal{E}(\phi'(f),f)
\end{equation}
holds for all functions $f\ge 0$, where $\mathcal{E}(f,g)$ is the Dirichlet form associated to $\L_h$ (see \eqref{2.diri}), a formal computation shows that 
\begin{align*}
  \frac{\dd}{\dd t}\mathcal{H}^\phi(S_tf|m_h)
  &= \sum_{i\in\K_h}\phi'(S_t f(i))\pa_t(S_t f)(i)m_h(i) 
  = \sum_{i\in\K_h}\phi'(S_t f(i))(\L_h S_tf)(i)m_h(i) \\
  &= -\mathcal{E}(\phi'(S_tf),S_tf)
  \le -\kappa_\phi\mathcal{H}^\phi(S_tf|m_h),
\end{align*}
and Gr\"onwall's inequality implies \eqref{1.res1}. To prove the convex Sobolev inequality \eqref{1.csi}, Bakry and Emery \cite{BaEm85} have shown that the second derivative $\dd^2\mathcal{H}^\phi/\dd t^2$ is bounded from below by the Dirichlet form, which implies \eqref{1.csi}. However, this task is very delicate in the discrete setting. 

Therefore, the second idea is to use the coupling method of Conforti \cite{Con22}. Roughly speaking, the idea is to compare two probability measures $\nu$ and $\eta$ and to construct a joint probability spaces with marginals $\nu$ and $\eta$. Conforti proved \eqref{1.res1} for various interacting random walks. We establish \eqref{1.res1} in the context of finite-volume schemes with strongly convex potentials and interpret the value $\kappa_\phi$ within this setting. In the proof, we use a synchronous coupling for the contraction rates with maximal probability. When a synchronous move is not possible, we couple the rates of the neighboring states such that the states after a jump are identical with maximal probability or are at least neighboring states.

Inequality \eqref{1.res2} is proved by using the path coupling method, see \cite{BuDy97}. This coupling approach reduces the computations in the proof only to neighbouring states. It holds for two continuous-time Markov chains $Y_t^1$ and $Y_t^2$ that
\begin{align*}
  \frac{\dd}{\dd t}\E[\dd(Y_t^1,Y_t^2)]
  = \E\big[\L_h^2(\dd(Y_t^1,Y_t^2))\big]
  \le -\kappa_1\E[\dd(Y_t^1,Y_t^2)],
\end{align*}
where d is the graph distance, $\L_h^2$ is the operator $\L_h$ on the product space $\K_h\times\K_h$, and $\kappa_1>0$ is related to the gap of the strictly diagonally dominance of the Hessian of $V$. We apply Gr\"onwall's inequality, take the minimum over all couplings, and use the fact that the graph and Euclidean distance are equivalent (with constant $\sqrt{d}$) to conclude \eqref{1.res2} from the previous inequality.

If the potential is additive, we obtain exponential contraction in the $L^2$ Wasserstein distance and, in one space dimension, in the $L^p$ Wasserstein distance with $p\ge 2$ (up to some numerical error). The contraction rate $\kappa_\phi$ in \eqref{1.res1} is of the order of the convexity constant $\kappa$. 

We prove similar results as \eqref{1.res1}--\eqref{1.res2} for discrete-time Markov chains associated to an explicit Euler scheme for the Fokker--Planck equation. Here, the idea is to apply the discrete Bakry--Emery method of \cite{JuSc17}, which is based on an estimate between the Fisher information $\mathcal{E}(\phi'(f),f)$ and the entropy production $-(\dd\mathcal{H}^\phi/\dd t)(S_t f|m_h)$ (which are the same in the continuous setting). We do not obtain contraction but exponential decay with a prefactor that depends on the initial datum $f$. 

\subsection{Comparison to the literature}

Based on the Bakry--Emery approach, $\phi$-entropy bounds for solutions to the Fokker--Planck equation are provided in \cite{BoGe10} from a probabilistic viewpoint and in \cite{AMTU01} from a PDE viewpoint. For discrete settings, the analysis becomes more involved because of the lack of a general (nonlinear) chain rule. We are aware of the following approaches. 

Caputo et al.\ used in \cite{CPP09} a new Bochner-type inequality, which replaces the Bochner identity of the continuous case, to prove a modified logarithmic Sobolev inequality for certain continuous-time Markov chains, including zero-range processes and Bernoulli--Laplace models. The idea to employ such an inequality was first presented in \cite{BCDP06}. The Bochner--Bakry--Emery method was extended by Fathi and Maas in \cite{FaMa16} in the context of Ricci curvature bounds and in \cite{JuYu17} to prove discrete Beckner inequalities, which were also derived in \cite{BoTe06} using an iteration method. 

Mielke investigated in \cite{Mie13} geodesic convexity properties of nonlocal transportation distances on probability spaces such that continuous-time Markov chains can be formulated as gradient flows. It is known that geodesic convexity implies exponential decay \cite{AGS05}. This idea was extended to nonlinear Fokker--Planck equations in \cite{CJS19}. 

A discrete Bakry--Emery method was suggested in \cite{JuSc17}. Unlike in the continuous case, the discrete entropy production and the Fisher information are distinguished and compared to each other. The Bakry--Emery method relies on an estimate of the production of the Fisher information, which requires discrete versions of integrations by parts and suitable chain rules. The nonlinear integration-by-parts formulas are ``translated'' to the discrete case by using the systematic integration-by-parts method of \cite{JuMa06}. This method allows for the treatment of nonlinear equations, but it seems to be restricted to numerical three-point schemes and hence to one-dimensional equations only.

Conforti established in \cite{Con22} a new probabilistic approach to prove convex Sobolev inequalities of the type \eqref{1.csi} and to quantify the exponential decay in terms of the entropy for continuous-time Markov chains. His approach uses coupling rates to treat the second time derivative of the entropy, thus avoiding the use of discrete Bochner-type inequalities. Recently, Pedrotti \cite{Ped23} has analyzed how contractive coupling rates can be used to prove stronger inequalities in the form of curvature lower bounds for Markov chains and geodesic convexity of entropy functionals. 

Maas and Matthes \cite{MaMa16} suggested a finite-volume discretization for a nonlinear parabolic fourth-order equation, which can be written as a gradient flow similarly as the Fokker--Planck equation \eqref{1.FP}. The long-time asymptotics was shown for additive potentials by taking advantage of the factorization of the problem. Canc\`es and Venel \cite{CaVe23} studied an implicit Euler finite-volume approximation of nonlinear Fokker--Planck equations using a discretization of the flux, which was motivated by probabilistic exclusion processes (and which have some similarity to our rates \eqref{1.c}). Their approach allows for non-additive potentials, but the long-time asymptotics was not investigated. The finite-volume scheme of \cite{ChHe20} allows for the proof of exponential equilibration, but the decay rate depends on the size of the domain (because of the use of the Poincar\'e inequality; also see \cite{FiHe17}), while we assume strong convexity of the potential. A link between finite-volume-type discretizations for the Fokker--Planck equation in Vorono\"{\i} meshes and gradient flows with discrete Wasserstein distances on graphs was established in \cite{AlMa17}. 

In this paper, we use Conforti's coupling method to quantify the exponential decay of finite-volume schemes allowing for non-additive potentials and with a decay rate only depending on the potential and not on the domain size.

This work is organized as follows. In Section \ref{sec.def}, the continuous-time Markov chain is defined and its unique invariant measure is determined. Our main results for continuous-time and discrete-time Markov chain are presented in Sections \ref{sec.main} and \ref{sec.main2}, respectively. We introduce Conforti's coupling method in Section \ref{sec.coupl} and prove the main results in Section \ref{sec.proofs} (for continuous-time Markov chains) and Section \ref{sec.proofs2} (for discrete-time Markov chains).


\section{Definition of the continuous-time Markov chain}\label{sec.def}

The state space is defined by $D=[-K,K]^d$ for $K>0$, $d\ge 1$. Let $h>0$ be such that $2K/h\in\N$, and let $e_1,\ldots,e_d$ be the Euclidean basis vectors of $\R^d$. The rectangular finite-volume mesh is given by a family of control volumes (cubes) 
$$
  N_i = [i_1-h/2,i_1+h/2]\times\cdots\times[i_d-h/2,i_d+h/2]
$$ 
with center $i=(i_1,\ldots,i_d)\in\K_h$, where $i_j=-K+hn_j-h/2$ for $n_j\in\{1,\ldots,2K/h\}$. Let $\sigma>0$ be a diffusion constant and $V\in C^2(\R^d)$ be a potential.

To define the continuous-time Markov chain, we recall definition \eqref{1.c} of the transition rates. The rate $c(i,\pm_j)$ defines the numerical flux from the cell $N_i$ to $N_{i\pm h e_j}$. The definition of $c(i,\pm_j)$ implies that the numerical flux at the boundary vanishes. The rates are path-independent in the sense
\begin{equation}\label{2.path}
  c(i,\pm_j)c(i\pm he_j,\pm_\ell) = c(i,\pm_\ell)c(i\pm he_\ell,\pm_j)
  \quad\mbox{for }i\in\K_h,\ j,\ell=1,\ldots,d.
\end{equation} 
This identity means that the rate of jumping first in direction $j$ and then in direction $\ell$ is the same as moving first in direction $\ell$ and then $j$. The discrete generator $\L_h$ can be written as
\begin{equation}\label{2.Lh}
  \L_h f(i) = \sum_{j=1}^d\big[c(i,+_j)\big(f(i+he_j)-f(i)\big)
  + c(i,-_j)\big(f(i-he_j)-f(i)\big)\big].
\end{equation}

With these preparations, we can define the continuous-time Markov chain $(Y_t^h)_{t\ge 0}$ on $\K_h$ via the discrete generator $\L_h$. Recall that $G_i$ denotes the set of all moves from $i$ to $i\pm he_j$ for $j=1,\ldots,d$, written as $\pm_j$. Let $(Z_n^h)_{n\in\N}$ be a discrete-time Markov chain on $\K_h$ with initial distribtion $\nu$, given by the transition kernel
\begin{align}\label{1.tk}
  \pi(i,\bar\imath) = \begin{cases}
  c(i,\gamma)/\mathcal{T} 
  &\mbox{if there exists }\gamma\in G_i\mbox{ such that }
  \bar\imath=\gamma i, \\
  1-\sum_{\gamma\in G_i}c(i,\gamma)/\mathcal{T} &\mbox{if }i=\bar\imath, \\
  0 &\mbox{else},
  \end{cases}
\end{align}
where $\mathcal{T}=2\max_{i\in\K_h}\sum_{\gamma\in G_i}c(i,\gamma)$. Let $T_n$ for $n\in\N$ be independent exponential random variables with parameter one, which are independent of $(Z_n^h)_{n\in\N}$. Set $S_n = \sum_{k=1}^n T_k/\mathcal{T}$. Then the continuous-time Markov chain is defined by
\begin{align}\label{2.Y}
  Y_t^h = \begin{cases}
  Z_n^h &\mbox{if }S_n\le t<S_{n+1}\mbox{ for some }n\in\N, \\
  \infty &\mbox{else}.
  \end{cases}
\end{align}
The corresponding transition function is denoted by $(p_t)_{t\ge 0}$. 

The discrete Fokker--Planck equation on $\K_h$ is given by 
$\pa_t u(t,i)=\L_h^* u(t,i)$, where $t\ge 0$, $i\in\K_h$, and the adjoint generator reads as
\begin{align}\label{2.Lh*}
  \L_h^* g(i) &= \sum_{j=1}^d\big(c(i-he_j,+_j)g(i-he_j)
  + c(i+he_j,-_j)g(i+he_j) \\
  &\phantom{xx}- c(i,+_j)g(i) - c(i,-_j)g(i)\big), \nonumber
\end{align}
with homogeneous Neumann boundary conditions. We observe that if $V$ is additive, i.e.\ $V(i)=\sum_{j=1}^d V_j(i_j)$ for $i\in\K_h$ (which allows for a factorization of the invariant measure), the finite-volume scheme and the generator coincide with the scheme and generator considered in \cite{MaMa16}.

The following lemma determines the invariant measure associated to the Markov chain.

\begin{lemma}\label{lem.invmeas}
The unique invariant measure $m_h$ of the continuous-time Markov chain $(Y_t^h)_{t\ge 0}$ with the transition rates \eqref{1.c} is given by
\begin{equation}\label{2.mh}
  m_h(i) = Z^{-1}\exp(-V^h(i)/\sigma^2) \quad
  \mbox{for }i\in\K_h
\end{equation}
with the normalization constant $Z=\sum_{i\in\K_h}\exp(-V^h(i)/\sigma^2)$.
\end{lemma}

\begin{proof}
The result follows immediately from the fact that the continuous-time Markov chain $(Y_n^h)_{n\in\N}$ satisfies the detailed-balance condition $c(i,+_j)m_h(i)=c(i+he_j,-_j)m_h(i+he_j)$ for all $i\in\K_h$ and $j=1,\ldots,d$, which implies  both the reversibility of the continuous-time Markov chain starting with initial measure $m_h$ and that $m_h$ given in \eqref{2.mh} is the unique invariant measure of $(Y_t^h)_{t\ge 0}$.

Alternatively, since the state space is finite and the Markov chain is irreducible and aperiodic, there exists a unique invariant measure $m_h$ and we can verify that $\sum_{i\in\K_h}\L_h f(i)m_h(i)=0$ holds for all positive functions $f$ on $\K_h$, which shows that \eqref{2.mh} is the invariant measure.
\end{proof}

\begin{remark}[Choice of $c(i,\gamma)$]\label{rem.c}\rm
By Definition \eqref{2.mh} of the invariant measure, we can write the discrete generator as
\begin{align*}
  &\L_h f(i) = \frac{\sigma^2}{h^2 Z}\sum_{j=1}^d
  \Big(\frac{M_h(i,+_j)}{m_h(i)}(f(i+he_j)-f(i)) 
  + \frac{M_h(i,-_j)}{m_h(i)}(f(i-he_j)-f(i))\Big), \\
  &\mbox{where } M_h(i,\pm_j) = \sqrt{m_h(i\pm he_j)m_h(i)}.
\end{align*}
The mean value $M_h(i,\pm_j)$ corresponds to the geometric mean used in  \cite[Cor.~5.5]{Mie13} (also see \cite[Sec.~4]{JuYu17}). Other choices for $M_h(i,\pm_j)$ can be found in \cite[Sec.~3.2]{Gli11} and \cite[Sec.~2.1.3]{ChHe20}. We observe that the difference of the flows to and from the neighboring cells satisfies
$$
  c(i,+_j)-c(i+he_j,-_j) = 2\frac{\sigma^2}{h^2}
  \sinh\bigg(\frac{V^h(i)-V^h(i+he_j)}{2\sigma^2}\bigg).
$$
This is related to the stochastic jump process in \cite[Sec.~1.2]{BoPe16}
and to the cosh structure of the dissipation potential of generalized gradient flows from \cite{LMPR17}. In \cite{BoPe16}, the transition rates of the one-dimensional Markov jump process depend only on the energy decay of the current state and do not take into account the energy difference between the current and the next state. Therefore, carrying these transition rates to multiple space dimensions, the path-independence property \eqref{2.path}, which is essential in our analysis, does not hold in general.
\qed\end{remark}


\section{Main results for continuous-time Markov chains}\label{sec.main}

We present the main results for continuous-time Markov chains. The proofs will be given in Section \ref{sec.proofs} and Appendix \ref{sec.conv}. 

\subsection{Convergence of the Markov chain approximation to the SDE}

Let $(Y_h^t)_{t\ge 0}$ be the continuous-time Markov chain solving the martingale problem with operator $\L_h$ defined in \eqref{2.Lh}, and let $(X_t)_{t\ge 0}$ be the solution to the SDE \eqref{1.sde}, which solves the martingale problem with operator $\L$ given by \eqref{1.L}. A Taylor expansion shows that $\L f(i)-\L_h f(i)=O(h)$ for $i\in\K_h$ and functions $f\in\mathcal{A}$, defined in \eqref{1.A}. This estimate is crucial in the proof that the Markov chain converges in distribution to the solution to \eqref{1.sde} as the grid size converges to zero, $h\to 0$. The following theorem is proved in Appendix \ref{sec.conv}.

\begin{theorem}[Convergence to SDE]\label{thm.conv}
Let $(Y_t^h)_{t\ge 0}$ be the continuous-time Markov chain with generator $\L_h$ on the grid $D\cap(h\Z)^d$ such that the laws of $Y_0^h$ converge in distribution to the measure $\mu_0$ as $h\to 0$. Then $(Y_t^h)_{t\ge 0}$ converges in distribution to the solution $(X_t)_{t\ge 0}$ to \eqref{1.sde} with initial datum $\mu_0$.
\end{theorem}

To construct a continuous-time Markov chain that converges to the solution to the SDE $\dd X_t = -\na V(X_t)\dd t + \sqrt{2\sigma}\dd B_t$ in the whole space, with $(B_t)_{t\ge 0}$ being a Brownian motion, we need to modify the cube $D=[-K,K]^d$. In fact, we replace $K$ by $K_h$ for $h>0$ such that $K_h\to\infty$ as $h\to 0$. Then Theorem \ref{thm.conv} holds similarly after adapting the diffusion approximation of \cite[Theorem 7.4.1]{EtKu86}.


\subsection{Exponential decay in $\phi$-entropy}

Consider the Markov chain $(Y_t^h)_{t\ge 0}$ on $\K_h$, characterized by the generator \eqref{2.Lh}. Let $\phi:\R_+\to\R_+$ and $f,g:\K_h\to\R_+$ be functions, where $\R_+:=[0,\infty)$. We recall the definition of the $\phi$-entropy, relative to the invariant measure $m_h$:
\begin{equation}\label{2.Hphi}
  \mathcal{H}^\phi(f|m_h) = \sum_{i\in\K_h}\phi(f(i))m_h(i)
  - \phi\bigg(\sum_{i\in\K_h}f(i)m_h(i)\bigg).
\end{equation}
We also introduce the Dirichlet form
\begin{equation}\label{2.diri}
  \mathcal{E}(f,g) = -\sum_{i\in\K_h}f(i)(\L_hg)(i)m_h(i)
  = \frac12\sum_{i\in\K_h}\sum_{\gamma\in G}c(i,\gamma)
  \na_\gamma f(i)\na_\gamma g(i) m_h(i),
\end{equation}
where we have set $\na_\gamma f(i)=f(\gamma i)-f(i)$. The last identity follows from the reversibility of the Markov chain (see, e.g., \cite[(2.12)]{DaPo13}). 

Our aim is to find the optimal constant $\kappa_\phi$ such that the exponential decay \eqref{1.res1} in the $\phi$-entropy holds for all functions $f:\K_h\to\R_+$ and $t>0$, where $S_t$ is the Markovian semigroup generated by $\L_h$. For continuous-time processes, this problem is known to be equivalent to find the optimal constant for the convex Sobolev inequality \eqref{1.csi}. Introduce the family of functions
\begin{align}\label{3.phia}
  \phi_\alpha(x) = \begin{cases}
  (\alpha-1)^{-1}(x^\alpha-x)-x+1 &\mbox{if }1<\alpha\le 2, \\
  x\log x-x+1 &\mbox{if }\alpha=1.
  \end{cases}
\end{align}
Then the convex Sobolev inequality becomes the modified logarithmic Sobolev inequality for $\alpha=1$, the Poincar\'e inequality for $\alpha=2$, and the Beckner inequality for $1<\alpha<2$.

We need the following assumptions.

\begin{itemize}
\item[(A1)] Convexity of entropy: Let the functions $\phi\in C^1(\R_+;\R_+)$ and $\Phi:(0,\infty)^2\to\R_+$, defined by $\Phi(a,b)=(\phi'(a)-\phi'(b))(a-b)$, be convex.
\item[(A2)] Strong $\kappa$-convexity: Let $V\in C^1(D;\R)$ be such that there exists $\kappa>0$ such that $\langle x-y,\na V(x)-\na V(y)\rangle\ge\kappa|x-y|^2$ for all $x,y\in D$, where $\langle\cdot,\cdot\rangle$ is the inner product on $\R^d$.
\item[(A3)] It holds that for all $i,i+he_j\in\K_h$ and $j=1,\ldots,d$,
\begin{equation}\label{3.kappa+}
  \kappa_+(i,j) := c(i,+_j)-c(i+he_j,+_j) - \sum_{\gamma\in G\setminus
  \{\pm_j\}}\max\{c(i+he_j,\gamma)-c(i,\gamma),0\} > 0,
\end{equation}
and for all $i,i-he_j\in\K_h$ and $j=1,\ldots,d$,
\begin{equation}\label{3.kappa-}
  \kappa_-(i,j) := c(i,-_j)-c(i-he_j,-_j) - \sum_{\gamma\in G\setminus
  \{\pm_j\}}\max\{c(i-he_j,\gamma)-c(i,\gamma),0\} > 0.
\end{equation}
\end{itemize}

\begin{theorem}[Exponential decay]\label{thm.decay}
Let Assumptions (A1) and (A3) hold. Consider the continuous-time Markov chain $(Y_t^h)_{t\ge 0}$ on $\K_h$ with transition rates \eqref{1.c}, which satisfy \eqref{2.path}. Then the convex Sobolev inequality \eqref{1.csi} holds with constant
\begin{equation*}
  \kappa_\phi = \min_{i,i+he_j\in \K_h,\, j=1,\ldots, d} 
  \{\kappa_+(i,j)+\kappa_-(i+he_j,j)\},
\end{equation*} 
and the exponential decay \eqref{1.res1} is valid. Moreover, the modified logarithmic Sobolev inequality holds with $\kappa_1=2\kappa_\phi$ and the discrete Beckner inequality holds with $\kappa_\alpha=\alpha\kappa_\phi$ for $\alpha\in(1,2]$. 
\end{theorem}

Compared to \cite{MaMa16}, the coupling approach allows us to show exponential decay in the $\phi$-entropy for more general (non-additive) potentials. We observe that the constant $\kappa_\phi$ is independent of the size of the domain $D$. Moreover, $\kappa_\phi$ scales linearly with the size of the transition rates. In particular, if we multiply the rates by some constant and change the rate when a jump occurs, $\kappa_\phi$ changes by this factor.

In the following, we discuss Assumptions (A1)--(A3). The functions $\phi_\alpha$ from \eqref{3.phia} and $\Phi(a,b)=(\phi'_\alpha(a)-\phi'_\alpha(b))(a-b)$ satisfy Assumption (A1). 

\begin{remark}[Additive potentials]\label{rem.add}\rm
For additive potentials $V(i)=\sum_{j=1}^d V_j(i_j)$ ($i\in\K_h$) that satisfy Assumption (A2), also Assumption (A3) is fulfilled. Indeed, for jumps $\gamma\neq \pm_j$, a computation shows that $c(i\pm he_j,\gamma)-c(i,\gamma)=0$. Then
\begin{align*}
  \kappa_+(i,j) &= c(i,+_j)-c(i,+ he_j,+_j) \\
  &= \frac{\sigma^2}{h^2}\big(e^{-(V_j^h(i_j+h)-V_j^h(i_j))/(2\sigma^2)}
  - e^{-(V_j^h(i_j+2h)-V_j^h(i_j+h))/(2\sigma^2)}\big) \\
  &= \frac{\sigma^2}{h^2}\exp\bigg(-\frac{1}{2\sigma^2}\int_0^h
  \pa V_j^h(i_j+s)\dd s\bigg) \\
  &\phantom{xx}\times\bigg[1-\exp\bigg(-\frac{1}{2\sigma^2}
  \int_0^h\big(\pa V_j^h(i_j+h+s)-\pa V_j^h(i_j+s)\big)\dd s
  \bigg)\bigg].
\end{align*}
Assumption (A2) implies for additive potentials that
$h(\pa V_j^h(i_j+h+s)-\pa V_j^h(i_j+s))\ge\kappa h^2$ and hence
$$
  \kappa_+(i,j)\ge\frac{\sigma^2}{h^2}
  \exp\bigg(-\frac{1}{2\sigma^2}\int_0^h\pa V_j^h(i_j+s)\dd s\bigg)
  \bigg[1-\exp\bigg(-\frac{1}{2\sigma^2}\int_0^h\kappa h\dd s\bigg)\bigg]
  > 0.
$$
It follows similarly that $\kappa_-(i,j)>0$, thus verifying Assumption (A3). We observe that the last expression converges to $\kappa/2$ as $h\to 0$. Thus, the decay rate $\kappa_\phi$ is asymptotically optimal in this situation.
\qed\end{remark}

\begin{remark}[Strictly diagonally dominance of $\mathrm{D}^2 V$]
\label{rem.A3}\rm
Assumption (A3) can be interpreted as a modified strictly diagonally dominant condition on the Hessian of $V$. Indeed, let $V\in C^2(D;\R)$ and let the Hessian $\mathrm{D}^2 V(i)\in\R^{d\times d}$ of $V$ be strictly diagonally dominant, i.e., there exists $c>0$ such that for all $i\in\K_h$ and $j=1,\ldots,d$,
$$
  (\mathrm{D}^2 V)_{jj}(i) 
  - \sum_{\ell\neq j}|(\mathrm{D}^2 V)_{\ell j}(i)| \ge c.
$$
Let $\pa_{\pm j}^h$ be the discrete derivative on $\K_h$ in direction $j$, defined by $\pa_{\pm j}^h f(i) = \pm(f(i+he_j)-f(i))/h$. A Taylor expansion yields
\begin{align*}
  c(i,+_\ell) - c(i+he_j,+_\ell) 
  &= \pa_{+j}^h\pa_{+\ell}^h V^h(i) + O(h), \\
  c(i,-_\ell) - c(i+he_j,-_\ell)
  &= \pa_{+j}^h\pa_{-\ell}^h V^h(i) + O(h), \mbox{ etc}.
\end{align*}
Then $\kappa_\pm(i,j)$ from Assumption (A3) can be approximated by
\begin{align*}
  \kappa_+(i,j) &= \pa_{+j}^h\pa_{+j}^h V^h(i)
  - \sum_{\ell\neq j}\big(\max\{\pa_{+j}^h\pa_{+\ell}^h V^h(i),0\}
  + \max\{-\pa_{+j}^h\pa_{-\ell}^h V^h(i),0\}\big) + O(h), \\
  \kappa_-(i,j) &= \pa_{-j}^h\pa_{-j}^h V^h(i)
  - \sum_{\ell\neq j}\big(\max\{\pa_{-j}^h\pa_{-\ell}^h V^h(i),0\}
  + \max\{-\pa_{-j}^h\pa_{+\ell}^h V^h(i),0\}\Big) + O(h).
\end{align*}
In the limit $h\to 0$, the discrete second derivatives of $V^h$ converge to the Hessian of $V$. Thus, if $\mathrm{D}^2 V$ is strictly diagonally dominant, Assumption (A3) is satisfied for sufficiently small $h>0$.
\qed\end{remark}


\subsection{Exponential decay in the Wasserstein distance}

Let $p\ge 1$ and let $\nu$, $\eta$ be two probability measures on $\K_h$ with finite $p$th moment. The $L^p$ Wasserstein distance between $\nu$ and $\eta$ with respect to the Euclidean distance is defined by
\begin{equation}\label{2.Wp}
  \mathcal{W}_p(\nu,\eta) = \inf_{\gamma\in\Gamma(\nu,\eta)}
  \bigg(\int_{\K_h\times\K_h}|x-y|^p\gamma(\dd x\dd y)\bigg)^{1/p},
\end{equation}
where $\Gamma(\nu,\eta)$ denotes the set of all couplings between $\nu$ and $\eta$, i.e.\ the set of all probability measures on $\mathcal{K}_h\times\mathcal{K}_h$ with marginals $\nu$ and $\eta$.
We first present results for additive potentials.

\begin{theorem}[Convergence in Wasserstein distance]\label{thm.W}
Let the potential $V\in C^2(\R^d)$ be of additive form, $V(i)=\sum_{j=1}^d V_j(i_j)$ for $i\in\K_h$, satisfy Assumption (A2), and possess a Lip\-schitz continuous gradient. Furthermore, let $\nu$ and $\eta$ be two probability measures on $\K_h$ and $(Y_t^h)_{t\ge 0}$ be the continuous-time Markov chain with transition function $(p_t)_{t\ge 0}$ associated to the transition rates \eqref{1.c}. Then it holds with the constant $\kappa>0$ from Assumption (A2) that
$$
  \mathcal{W}_2(\nu p_t,\eta p_t) 
  \le e^{-\kappa t}\mathcal{W}_2(\nu,\eta) + O(h^{1/2})
  \quad\mbox{for all }t\ge 0.
$$
Moreover, in one space dimension and for $p\ge 2$,
$$
  \mathcal{W}_p(\nu p_t,\eta p_t) 
  \le e^{-\kappa t}\mathcal{W}_p(\nu,\eta) + O(h^{1/p})
  \quad\mbox{for all }t\ge 0.
$$
\end{theorem}

The error term $O(h^{1/p})$ can be avoided by relaxing the contraction rate to $\kappa-O(h)$ for sufficiently small $h>0$. If the potential is not of additive form, we can still conclude long-time guarantees in the $L^1$ Wasserstein distance.

\begin{theorem}[Convergence in $L^1$ Wasserstein distance]\label{thm.W1}
Let the potential satisfy Assumption 
(A3). Furthermore, let $\nu$ and $\eta$ be two probability measures on $\K_h$ and $(Y_t^h)_{t\ge 0}$ be the continuous-time Markov chain with transition function $(p_t)_{t\ge 0}$ associated to the transition rates \eqref{1.c}. Then it holds that
$$
  \mathcal{W}_1(\nu p_t,\eta p_t) 
  \le \sqrt{d}e^{-\kappa_1 t}\mathcal{W}_1(\nu,\eta)
  \quad\mbox{for all }t\ge 0,
$$
where $d$ is the dimension of the state space and 
\begin{equation}\label{2.kappa1}
  \kappa_1 = \min_{i\in\K_h,\,j=1,\ldots,d}
  \big(\kappa_+(i,j)+\kappa_-(i+he_j,j)\big).
\end{equation}
\end{theorem}

\begin{remark}[Discussion of Theorem \ref{thm.W1}]\rm
The contraction rate does not equal $\kappa>0$ from Assumption (A2) but is generally smaller. This is due to the fact that we use a rectangular mesh in the approximation. In the case of additive potentials, the value $\kappa_1$ is of order $\kappa$. The prefactor $\sqrt{d}$ originates from the change of the Wasserstein distance $\mathcal{W}_{\dd,1}$ with respect to the graph distance to the Euclidean Wasserstein distance $\mathcal{W}_1$. Indeed, we prove first, using the path coupling method of \cite{BuDy97}, that
\begin{equation}\label{3.Wd1}
  \mathcal{W}_{\dd,1}(\nu p_t,\eta p_t) \le e^{-\kappa_1 t}
  \mathcal{W}_{\dd,1}(\nu,\eta),
\end{equation}
where $\dd(x,y)=\sum_{j=1}^d|x_j-y_j|$ for $x,y\in\R^d$ is the graph distance. Our result then follows from the fact that the graph distance is equivalent to the Euclidean distance, resulting in the prefactor $\sqrt{d}$. 
\qed\end{remark}

\begin{remark}[Finite-difference scheme]\rm
The finite-difference discretization of \eqref{1.FP} reads as
$$
  \pa_t u(t,i) = \L_h^* u(t,i) = \sigma^2\Delta_h u(t,i)
  + \diver_h(u(t,i)\na V(t,i)),
$$ 
where the discrete Laplacian and divergence are defined by
\begin{align*}
  \Delta_h f(i) &= \frac{1}{h^2}\sum_{j=1}^d\big(f(i+he_j) - 2f(i)
  + f(i-he_j)\big), \\
  \diver_h g(i) &= \frac{1}{2h}\sum_{j=1}^d\big(g_j(i+he_j)-g_j(i-he_j)
  \big)
\end{align*}
for $i\in\K_h$ and functions $f:\K_h\to\R$ and $g:\K_h\to\R^d$. The generator of the associated Markov chain reads as \eqref{1.L}, but the jump rates are now given by
\begin{equation*}
  c(i,\pm_j) = \begin{cases}
  h^{-2}(\sigma^2\mp h\pa_j V(i)/2) &\mbox{if }i+he_j,i-he_j\in\K_h, \\
  h^{-2}(2\sigma^2\mp h\pa_j V(i)) & \mbox{if }i\mp he_j\not\in\K_h, \\
  0 &\mbox{if }i\pm he_j\not\in\K_h.
  \end{cases}
\end{equation*}
At the boundary $\pa\K_h$, the process is reflected, which corresponds at the level of the Fokker--Planck equation to homogeneous Neumann boundary conditions. To ensure the positivity of the rates, we need to impose the general assumption
\begin{equation}\label{3.fda}
  \inf_{x\in D,\,j=1,\ldots,d}\bigg(\sigma^2 - \frac{h}{2}|\pa_j V(x)|
  \bigg) \ge 0.
\end{equation}
We observe that the rates given above are a first-order approximation of the jump rates \eqref{1.c}. As for the finite-volume scheme, we define the continuous-time Markov chain $(Y_t^h)_{t\ge 0}$ on $\K_h$ via the discrete generator $\L_h$ by \eqref{2.Y}. Then $(Y_t^h)_{t\ge 0}$ is non-explosive, irreducible, and recurrent. 

For additive potentials $V(x)=\sum_{j=1}^d V_j(x_j)$, a computation shows that the invariant measure of $(Y_t^h)_{t\ge 0}$ equals
$$
  m_h(i) = \frac{1}{Z}\prod_{j=1}^d\prod_{\ell\in\K_h,\,\ell<i_j}
  \frac{\sigma^2-h\pa V_j(\ell)/2}{\sigma^2 + h\pa V_j(\ell+h)/2},
$$
where $Z>0$ is a normalization constant. For non-additive potentials and multiple space dimensions, the structure of the invariant measure becomes more complicated compared to the finite-volume scheme and we leave details to the reader.

The finite-difference scheme has the drawback that the rates are not commutative for more than one space dimension in the sense that \eqref{2.path} does not hold. However, assuming that the potential is additive, it is possible to carry over the results stated for the finite-volume scheme to the finite-difference discretization (supposing condition \eqref{3.fda}). Again, we leave the details to the reader.
\qed\end{remark}


\section{Main results for discrete-time Markov chains}\label{sec.main2}

\subsection{Exponential decay in $\phi$-entropy}

We consider a discrete-time Markov chain as an approximation of the solution to the SDE \eqref{1.sde}. Instead of exponentially distributed random times, we choose equidistant time steps, which means that for fixed $\tau>0$, the Fokker--Planck equation is discretized according to the explicit Euler finite-volume scheme
$$
  \frac{1}{\tau}\big(u(t_{n+1},i) - u(t_n,i)\big) 
  = \L_h^* u(t_n,i)\quad\mbox{for }i\in\K_h,\ t_n=n\tau,\ n\in\N,
$$
where $\L_h^*$ is defined in \eqref{2.Lh*}. The corresponding Markov chain $(Z_n^h)_{n\in\N}$ is defined through the transition kernel \eqref{1.tk}. Interpreting $\pi$ as the matrix with entries $\pi(i,k)$, we have
$$
  \pi f(i) = \sum_{k\in\K_h}\pi(i,k)f(k).
$$
With the choice of the rescaling factor $\mathcal{T}=2\max_{i\in\K_h}\sum_{\gamma\in G_i} c(i,\gamma)$, we obtain a lazy random walk, where we remain at the current position with probability at least 1/2. Furthermore, we introduce the rescaled jump rates
\begin{equation}\label{3.p}
  p(i,\gamma) = \frac{c(i,\gamma)}{\mathcal{T}}\quad\mbox{for }
  \gamma\in G_i.
\end{equation}
and we set $\tau:=\mathcal{T}^{-1}$. 

In the situation of Lemma \ref{lem.invmeas}, the unique invariant measure of the discrete-time Markov chain $(Z_n^h)_{n\in\N}$ with transition kernel \eqref{1.tk} is still given by \eqref{2.mh}. Indeed, the lazyness of the random walk ensures the aperiodicity of the Markov chain, which yields the uniqueness of an invariant measure and moreover, it holds for all functions $f:\K_h\to\R$ that
\begin{equation}\label{2.inv}
  \sum_{i\in\K_h}\pi f(i)m_h(i)
  = \sum_{i\in\K_g}\sum_{\gamma\in G}p(i,\gamma)f(\gamma i)m_h(i)
  = \sum_{i\in\K_h}f(i)m_h(i).
\end{equation}

Unfortunately, we cannot use the convex Sobolev inequality \eqref{1.csi} to prove the exponential decay in the $\phi$-entropy, since the discrete time derivative of the $\phi$-entropy does not correspond to the Dirichlet form. We overcome this issue by applying the abstract method of \cite{JuSc17}, which distinguishes the entropy production and the Fisher information (which both equal the Dirichlet form in the continuous case). Interestingly, while the approach of \cite{JuSc17} is based on the implicit Euler scheme, we are able to analyze the explicit Euler scheme.

We introduce for $f\ge 0$ the (discrete) entropy production $\mathcal{P}(f)$ and Fisher information $\mathcal{F}(f)$, respectively, by
\begin{align*}
  \mathcal{P}(f) = -\tau^{-1}\big(\mathcal{H}^\phi(\pi f|m_h) 
  - \mathcal{H}^\phi(f|m_h)\big), \quad
  \mathcal{F}(f) = \mathcal{E}(\phi'(f),f).
\end{align*}
The following results are proved in Section \ref{sec.proofs2}.

\begin{proposition}\label{prop.BE}
We assume that
\begin{itemize}
\item[(i)] There exists $C_P>0$ such that $0\le \mathcal{P}(\pi^n f) \le C_P\mathcal{F}(\pi^n f)$ for all $f\ge 0$ and $n\in\N$.
\item[(ii)] There exists $\lambda>0$ such that $\mathcal{F}(\pi^{n+1}f)-\mathcal{F}(\pi^nf)\le -\tau\lambda\mathcal{F}(\pi^n f)$ for all $f\ge 0$ and $n\in\N$.
\item[(iii)] $\lim_{n\to\infty}\mathcal{H}^\phi(\pi^nf|m_h)=0$.
\end{itemize}
Then for all $f\ge 0$, $n\in\N$, and $\tau<1/\lambda$,
$$
  \mathcal{H}^\phi(\pi^n f|m_h) 
  \le C_f e^{-\lambda n\tau}\mathcal{H}^\phi(f|m_h),
$$
where $C_f=C_P\mathcal{F}(f)/(\lambda\mathcal{H}^\phi(f|m_h))$.
\end{proposition}

We impose a weaker condition compared to \cite[Prop.~1]{JuSc17}, but we obtain a weaker result. In particular, the prefactor $C_f$ depends on $f$ and may be very large. This prefactor equals one in \cite{JuSc17}, thus providing a contraction result, but assuming the lower bound $c_P\mathcal{F}(\pi^n f)\le\mathcal{P}(\pi^n f)$ for some $c_P>0$. Unfortunately, this bound does not hold in our situation. It is not surprising that our result is weaker than in \cite{JuSc17}, since we consider an explicit scheme, while an implicit scheme is studied in \cite{JuSc17}.

Using the same coupling approach as for the continuous-time Markov chain, we conclude the exponential decay for the discrete-time Markov chain.

\begin{theorem}[Exponential decay]\label{thm.ddecay}
Let Assumption 
(A3) hold and let the dis\-crete-time Markov chain with transition matrix \eqref{1.tk} be given. Then, for $\phi_\alpha$ defined in \eqref{3.phia}, there exists $C_f^\alpha>0$ such that for $f\ge 0$, $n\in\N$, and $\tau<1/\kappa_{\phi}$,
\begin{align*}
  \mathcal{H}^{\phi_\alpha}(\pi^n f|m_h) 
  \le C^\alpha_f e^{-\kappa_{\phi} n\tau}\mathcal{H}^{\phi_\alpha}(f|m_h),
\end{align*}
where $C^\alpha_f>0$ is as in Proposition \ref{prop.BE} with $\phi$ replaced by $\phi_\alpha$, and $\kappa_{\phi}$ is as in Theorem \ref{thm.decay}.
\end{theorem}


\subsection{Exponential decay in the Wasserstein distance}

Analogously to the continuous-time situation, we obtain contraction in the Wasserstein distance up to some error term.

\begin{theorem}[Convergence in Wasserstein distance]\label{thm.dW}
Let the potential $V\in C^2(\R^d)$ be of additive form, $V(i)=\sum_{j=1}^d V_j(i_j)$ for $i\in\K_h$, satisfy Assumption (A2), and possess a Lip\-schitz continuous gradient. Furthermore, let $\nu$ and $\eta$ be two probability measures on $\K_h$, and consider the discrete-time Markov chain with transition matrix \eqref{1.tk}. Then it holds with the constant $\kappa>0$ from Assumption (A2) that
$$
  \mathcal{W}_2(\nu\pi^n,\eta\pi^n) \le e^{-\kappa n\tau}
  \mathcal{W}_2(\nu,\eta) + O(h^{1/2}) \quad\mbox{for all }n\in\N.
$$
Moreover, if the potential satisfies additionally Assumption (A3),
\begin{equation} \label{e.Wd}
  \mathcal{W}_{\dd, 1}(\nu\pi^n,\eta\pi^n) \le e^{-\kappa_1 n\tau}
  \mathcal{W}_{\dd,1} (\nu,\eta)\quad\mbox{for all }n\in\N,
\end{equation}
and
$$
  \mathcal{W}_1(\nu\pi^n,\eta\pi^n) \le \sqrt{d}e^{-\kappa_1 n\tau}
  \mathcal{W}_1(\nu,\eta)\quad\mbox{for all }n\in\N,
$$
where $\kappa_1>0$ is defined in \eqref{2.kappa1}.
\end{theorem}

The theorem is proved in the same way as Theorems \ref{thm.W} and \ref{thm.W1}.

\begin{remark}[Coarse Ricci curvature]\rm
If the transition kernel $\pi$ on $\K_h$ satisfies \eqref{e.Wd}, we say that the tripel $(\K_h,\pi,\dd)$ has a coarse Ricci curvature at least $\kappa_1\tau$ in the sense of Ollivier \cite{Oll09}. Furthermore, according to \cite{ELL17}, the decay \eqref{3.Wd1} implies that
$$
  \mathcal{W}_{\dd,1}(\nu,m_h) 
  \le \bigg(\frac{2\mathcal{H}^{\phi_1}(\nu|m_h)}{\kappa_1\tau(2-\kappa_1\tau)}
  \bigg)^{1/2},
$$
where $\mathcal{H}^{\phi_1}(\nu|m_h)$ denotes the relative entropy between two probability measures, which equals $\mathcal{H}^{\phi_1}(\rho|m_h)$ with $\rho=\dd\nu/\dd m_h$ if $\nu$ is absolutely continuous with respect to $m_h$, and $+\infty$ else. Theorem \ref{thm.decay} guarantees the validity of the modified logarithmic Sobolev inequality with constant $\kappa_\phi$. Peres and Tetaly have conjectured for Markov chains on discrete spaces, for which a coarse Ricci curvature in the sense of Ollivier holds, that a modified logarithmic Sobolev inequality with the same constant up to a constant factor holds; see \cite[Conjecture 3.1]{ELL17}. This conjecture has been recently disproven in \cite{Mue23}. However, by \cite[Theorem 1]{CMS25}, contraction in Wasserstein distance implies directly contraction in entropy with the same constant, i.e., $\mathcal{H}^{\phi_1}(\pi f|m_h)\le (1-\kappa)\mathcal{H}^{\phi_1}(f|m_h)$ provided an additional non-negative sectional curvature holds (recall that $\phi_1(x)=x\log x$ is the Boltzmann entropy). This result leads to a refinement of the standard modified logarithmic Sobolev inequality.
\qed\end{remark}

As a consequence of Theorem~\ref{thm.dW} and in complement to Theorem~\ref{thm.ddecay}, exponential convergence in entropy is obtained for additive potentials $V$ with improved rates by applying \cite[Theorem 1]{CPP09}.

\begin{corollary}
 Let the potential $V\in C^2(\R^d)$ be of additive form, $V(i)=\sum_{j=1}^d V_j(i_j)$ for $i\in\K_h$, satisfy Assumptions (A2) and (A3), and possess a Lip\-schitz continuous gradient. Let the dis\-crete-time Markov chain with transition matrix \eqref{1.tk} be given. Then, for all $f\ge 0$, $n\in \mathbb{N}$ it holds that
\begin{align*}
  \mathcal{H}^{\phi_1}(\pi^n f|m_h) 
  \le e^{-\kappa_{1} n \tau }\mathcal{H}^{\phi_1}(f|m_h),
\end{align*}
where $\kappa_1$ is defined in \eqref{2.kappa1}.
\end{corollary}


\section{Coupling method}\label{sec.coupling}\label{sec.coupl}

We introduce the coupling rates for Markov chains and present sufficient conditions to conclude the convex Sobolev inequality. These conditions were established by Conforti \cite{Con22}, and we recall them for the convenience of the reader. More precisely, we define a coupling of two Markov chains by coupling their rates. Recall that $G_i$ is the set of all possible moves of the Markov chain starting in $i\in\K_h$. We set $G_i^*=G_i\cup\{e\}$ for the set of all moves starting from $i$ and the null element $e$ (no move). 

\begin{definition}[Coupling rate]
Let $\L_h$ be the generator \eqref{2.Lh} with transition rates $c(i,\gamma)$. Given $i,\bar{\imath}\in\K_h$, we call the function $\mathbf{c}(i,\bar{\imath},\cdot,\cdot):G_i^*\times G_{\bar\imath}^*\to\R_+$ a {\em coupling rate} for $(i,\bar{\imath})$ if and only if
\begin{align*}
  \sum_{\bar\gamma\in G^*_{\bar{\imath}}}
  \mathbf{c}(i,\bar{\imath},\gamma,\bar\gamma)
  = c(i,\gamma) &\quad\mbox{for all }\gamma\in G_i, \\
  \sum_{\gamma\in G^*_{i}}\mathbf{c}(i,\bar{\imath},\gamma,\bar\gamma)
  = c(\bar{\imath},\bar\gamma) &\quad\mbox{for all }
  \bar\gamma\in G_{\bar{\imath}}.
\end{align*} 
\end{definition}

As in \cite{Con22}, we introduce the nonnegative function
$$
  f^\phi(i,\delta i) = \Phi(f(i),f(\delta i))
  = \big(\phi'(f(i))-\phi'(f(\delta i))\big)(f(i)-f(\delta i))
$$
for $i\in\K_h$, $\delta\in G_i$, and $f\ge 0$. Furthermore, let $S=\{(i,\delta)\in\K_h\times G:c(i,\delta)>0\}$ be the set of all combinations of states and moves with positive transition rate.

\begin{theorem}{\cite[Prop.~2.1]{Con22}}\label{thm.conforti}
Let Assumption (A1) hold. Let $c(i,\gamma i,\cdot,\cdot)$ be coupling rates and $m_h$ be the invariant measure associated to the generator $\L_h$.
\begin{itemize}
\item[(i)] If there exists $\kappa'\ge 0$ such that
\begin{equation}\label{4.kap}
  \frac12\sum_{\substack{(i,\delta)\in S \\ 
  \gamma\in G_i^*,\,\bar\gamma\in G^*_{\bar{\imath}}}}
  c(i,\gamma)\mathbf{c}(i,\delta i,\gamma,\bar\gamma)
  \big(f^\phi(\gamma i,\bar\gamma\delta i)-f^\phi(i,\delta i)\big)
  m_h(i) \le -\kappa'\mathcal{E}(\phi'(f),f)
\end{equation}
holds uniformly for $f>0$, the convex Sobolev inequality holds with constant $\kappa'$.
\item[(ii)] If in addition to \eqref{4.kap} there exists $\kappa''>0$ such that
$$
  \inf_{(i,\delta)\in S}\min\big\{\mathbf{c}(i,\delta i,\delta,e),
  \mathbf{c}(i,\delta i,e,\delta^{-1})\big\} \ge \kappa'',
$$
the modified logarithmic Sobolev inequality holds with $\kappa_1=\kappa'+2\kappa''$.
\item[(iii)] If in addition to \eqref{4.kap} there exists $\kappa'''>0$ such that
$$
  \inf_{(i,\delta)\in S}\sum_{\substack{\gamma\in G^*_i,\,
  \bar\gamma\in G^*_{\delta i} \\ \gamma i=\bar\gamma\delta i}}
  \mathbf{c}(i,\delta i,\gamma,\bar\gamma) \ge\kappa''',
$$
the discrete Beckner inequality holds for $1<\alpha\le 2$ with $\kappa_\alpha = \kappa' + (\alpha-1)\kappa'''$.
\end{itemize}
\end{theorem}

The proof of this theorem relies on the following construction of the coupling rates, which is also used in the proof of Theorem \ref{thm.decay} (see Figure~\ref{figure1} for an illustration of the coupling). We define for $i,i+he_j\in\K_h$ with $j=1,\ldots,n$,
\begin{equation}\label{4.coupl}
  \mathbf{c}(i,i+he_j,\gamma,\bar{\gamma})
  = \begin{cases} 
  \min\{c(i,\gamma),c(i+he_j,\bar{\gamma})\} 
  & \mbox{if }\gamma=\bar{\gamma}\in G, \\
  \max\{c(i+he_j,\bar{\gamma})-c(i,\bar{\gamma}),0\}
  & \mbox{if }\gamma=+_j, \bar{\gamma}\in G,\, 
  \bar{\gamma}\neq +_j,-_j, \\
  \max\{-c(i+he_j,\gamma)+c(i,\gamma),0\} 
  & \mbox{if }\gamma\in G,\, 
  \gamma\neq +_j,-_j,\, \bar{\gamma}=-_j, \\
  \kappa_+(i,j) 
  & \mbox{if } \gamma=+_j,\, \bar{\gamma}=e, \\
  \kappa_-(i+he_j,j) 
  & \mbox{if } \gamma=e,\, \bar{\gamma}=-_j, \\
  0 & \text{else,}
\end{cases}
\end{equation}
recalling definitions \eqref{3.kappa+} and \eqref{3.kappa-} of $\kappa_\pm(i,j)$. Assumption (A3) guarantees that $\kappa_\pm(i,j)>0$. The coupling ensures that the new states $\gamma i$ and $\bar\gamma(i+he_j)$ are at most neighboring states on the grid and that they are coupled with maximal probability, i.e.\ $\gamma i=\bar\gamma(i+he_j)$ with rate $\kappa_+(i,j)+\kappa_-(i,j)$. Analogously, we define
\begin{equation*}
  \mathbf{c}(i,i-he_j,\gamma,\bar\gamma) = \mathbf{c}(i-he_j,(i-he_j)+he_j,\bar\gamma,\gamma)\quad\mbox{for }i,i-he_j\in\K_h,\ j=1,\ldots,d.
\end{equation*}

\begin{figure}

\begin{tikzpicture}
\tikzstyle{every node}=[font=\tiny]
\coordinate[label=below:  ] (A) at (0,0);
\coordinate[label=below: ] (B) at (1,0);
\coordinate[label=below:$i$] (C) at (0,1);
\coordinate[label=below:$i+he_j$] (D) at (1,1);
\coordinate[label=left:] (F) at (-1,1);
\coordinate[label=right:] (H) at (2,1);
\coordinate[label=left:$i+he_k$] (J) at (0,2);
\coordinate[label=right:$i+h(e_j+e_k)$] (K) at (1,2);
\fill (A) circle (2pt);
\fill (B) circle (2pt);
\fill (C) circle (2pt);
\fill (D) circle (2pt);
\fill (F) circle (2pt);
\fill (H) circle (2pt);
\fill (J) circle (2pt);
\fill (K) circle (2pt);
\draw[->, red, very thick, dashed] (C) to node[left] {\ding{192}} (J);
\draw[->, blue, very thick] (D) to node[right] {\ding{193}} (K);
\draw[->, blue, very thick] (D) to node[above] {\ding{194}} (C);
\coordinate[label={[red]right: {\ding{192} $c(i, +_k)$}}] (M) at (-1.5,-0.5);
\coordinate[label={[blue]right: {\ding{193} $\min(c(i,+_k),c(i+he_j,+_k))$}} ] (N) at (-1.5,-0.8);
\coordinate[label={[blue]right: {\ding{194} $\max(c(i,+_k)-c(i+he_j,+_k),0)$}}] (O) at (-1.5,-1.1);

\coordinate[label=left: $i-he_k$] (A2) at (5,0);
\coordinate[label=right:$i+h(e_j-e_k)$] (B2) at (6,0);
\coordinate[label=left:$i$] (C2) at (5,1);
\coordinate[label=right:$i+he_j$] (D2) at (6,1);
\coordinate[label=left:] (F2) at (4,1);
\coordinate[label=right:] (H2) at (7,1);
\coordinate[label=above:] (J2) at (5,2);
\coordinate[label=above:] (K2) at (6,2);
\fill (A2) circle (2pt);
\fill (B2) circle (2pt);
\fill (C2) circle (2pt);
\fill (D2) circle (2pt);
\fill (F2) circle (2pt);
\fill (H2) circle (2pt);
\fill (J2) circle (2pt);
\fill (K2) circle (2pt);
\draw[->, red, very thick, dashed] (C2) to  node[left] {\ding{192}}(A2);
\draw[->, blue, very thick] (D2) to  node[right] {\ding{193}} (B2);
\draw[->, blue, very thick] (D2) to  node[above] {\ding{194}} (C2);

\coordinate[label={[red]right: {\ding{192} $c(i, -_k)$}}] (M) at (3.5,-0.5);
\coordinate[label={[blue]right: {\ding{193} $\min(c(i,-_k),c(i+he_j,-_k))$}} ] (N) at (3.5,-0.8);
\coordinate[label={[blue]right: {\ding{194} $\max(c(i,-_k)-c(i+he_j,-_k),0)$}}] (O) at (3.5,-1.1);

\coordinate[label=below: ] (A3) at (10,0);
\coordinate[label=below:] (B3) at (11,0);
\coordinate[label=above:$i$] (C3) at (10,1);
\coordinate[label=below:$i+he_j$] (D3) at (11,1);
\coordinate[label=below:$i-he_j$] (F3) at (9,1);
\coordinate[label=above:$i+2he_j$] (H3) at (12,1);
\coordinate[label=above:] (J3) at (10,2);
\coordinate[label=above:] (K3) at (11,2);
\fill (A3) circle (2pt);
\fill (B3) circle (2pt);
\fill (C3) circle (2pt);
\fill (D3) circle (2pt);
\fill (F3) circle (2pt);
\fill (H3) circle (2pt);
\fill (J3) circle (2pt);
\fill (K3) circle (2pt);
\draw[->, red, very thick, dashed] (C3) to node[above] {\ding{192}} (F3);
\draw[->, blue, very thick] (D3) to node[above] {\ding{193}} (C3);

\coordinate[label={[red]right: {\ding{192} $c(i, -_j)$}}] (M) at (8.5,-0.5);
\coordinate[label={[blue]right: {\ding{193} $\min(c(i,-_j),c(i+he_j,-_j))$}} ] (N) at (8.5,-0.8);

\tikzstyle{every node}=[font=\tiny]
\coordinate[label=below:  ] (A4) at (0,-4);
\coordinate[label=right:$i+h(e_j-e_k)$] (B4) at (1,-4);
\coordinate[label=left:$i$] (C4) at (0,-3);
\coordinate[label=below: ] (D4) at (1,-3);
\coordinate[label=below:$i+he_j$] (D4a) at (1.5,-3);
\coordinate[label=left: ] (F4) at (-1,-3);
\coordinate[label=right:$i+2he_j$] (H4) at (2,-3);
\coordinate[label=above: ] (J4) at (0,-2);
\coordinate[label=right:$i+h(e_j+e_k)$] (K4) at (1,-2);
\fill (A4) circle (2pt);
\fill (B4) circle (2pt);
\fill (C4) circle (2pt);
\fill (D4) circle (2pt);
\fill (F4) circle (2pt);
\fill (H4) circle (2pt);
\fill (J4) circle (2pt);
\fill (K4) circle (2pt);
\draw[->, red, very thick, dashed] (C4) to node[above] {\ding{192}} (D4);
\draw[->, blue, very thick] (D4) to node[above] {\ding{193}} (H4);
\draw[->, blue, very thick] (D4) to node[left] {\ding{194}} (B4);
\draw[->, blue, very thick] (D4) to node[left] {\ding{195}} (K4);
\draw[>= stealth, blue, very thick]  (D4) edge [out=80,in=30,distance=10mm]  node[above] {\ding{196}} (D4);

\coordinate[label={[red]right: {\ding{192} $c(i, +_j)$}}] (M) at (-1.5,-4.5);
\coordinate[label={[blue]right: {\ding{193} $\min(c(i,+_j),c(i+he_j,+_j))$} }] (N) at (-1.5,-4.8);
\coordinate[label={[blue]right: {\ding{194} $\max(c(i,+_k)-c(i+he_j,+_k),0)$}}] (O) at (-1.5,-5.1);
\coordinate[label={[blue]right: {\ding{195} $\max(c(i,-_k)-c(i+he_j,-_k),0)$}}] (O) at (-1.5,-5.4);
\coordinate[label={[blue]right: {\ding{196} $\kappa_+(i,j)$}}] (O) at (-1.5,-5.7);

\coordinate[label=left:   ] (A5) at (5,-4);
\coordinate[label=below: ] (B5) at (6,-4);
\coordinate[label=below:$i$] (C5) at (5,-3);
\coordinate[label=below:$i+he_j$] (D5) at (6,-3);
\coordinate[label=left:] (F5) at (4,-3);
\coordinate[label=above:$i+2he_j$] (H5) at (7,-3);
\coordinate[label=left:  ] (J5) at (5,-2);
\coordinate[label=above: ] (K5) at (6,-2);
\fill (A5) circle (2pt);
\fill (B5) circle (2pt);
\fill (C5) circle (2pt);
\fill (D5) circle (2pt);
\fill (F5) circle (2pt);
\fill (H5) circle (2pt);
\fill (J5) circle (2pt);
\fill (K5) circle (2pt);
\draw[->, blue, very thick] (D5) to  node[above] {\ding{193}} (C5);
\draw[>= stealth, red, very thick, dashed]  (C5) edge [out=240,in=120,distance=10mm]  node[left] {\ding{192}} (C5);

\coordinate[label={[red]right: {\ding{192} $c(i, e)$}}] (M) at (3.5,-4.5);
\coordinate[label={[blue]right: {\ding{193} $\kappa_-(i+he_j,j)$}} ] (N) at (3.5,-4.8);

\end{tikzpicture}
\caption{{\footnotesize Coupling for the jump rates of the states $i$ and $i+he_j$. For the five possible jumps for state $i$ illustrated by red dashed lines, the corresponding coupled jumps for state $i+he_j$ are given by blue solid lines.}} \label{figure1}
\end{figure}

We remark that if the potential has an additive structure, i.e.\ $V(i)=\sum_{j=1}^d V_j(i_j)$, coupling \eqref{4.coupl} simplifies to
\begin{equation*}
  \mathbf{c}(i,i+he_j,\gamma,\bar{\gamma})
  = \begin{cases} 
  \min\{c(i,\gamma),c(i+he_j,\bar{\gamma})\} 
  & \mbox{if } \gamma=\bar{\gamma}\in G, \\
  c(i,+_j)-c(i+he_j,+_j) & \mbox{if } \gamma=+_j, \bar{\gamma}=e, \\
  c(i+he_j,-_j)-c(i,-_j) & \mbox{if } \gamma=e, \bar{\gamma}=-_j, \\
  0 & \text{else.}
\end{cases}
\end{equation*}


\section{Proofs for continuous-time Markov chains}\label{sec.proofs}

\subsection{Exponential decay in $\phi$-entropy}

In this subsection, we prove Theorem \ref{thm.decay}. It is sufficient to verify the conditions of Theorem \ref{thm.conforti}. The most delicate condition is the first one, whose validity is proved in the following lemma.

\begin{lemma}[Condition (i)]\label{lem.ineq}
Let Assumptions (A2) and (A3) hold and let $f\ge 0$. Then, for $\kappa_\phi$ as given in Theorem \ref{thm.decay}, the following inequality holds:
\begin{equation}\label{4.aux}
  \sum_{\substack{(i,\delta)\in S \\ \gamma\in G^*_i,\, 
  \bar{\gamma}\in G^*_{\delta i}}} c(i,\delta) 
  \mathbf{c}(i,\delta i, \gamma, \bar{\gamma})
  \big(f^{\phi}(\gamma i,\bar{\gamma}\delta i)
  -f^{\phi}(i,\delta i)\big) m_h(i)
  \le -\kappa_{\phi} \sum_{(i,\delta)\in S}
  c(i,\delta)f^{\phi}(i,\delta i)m_h(i).
\end{equation}
\end{lemma}

Recall that $S$ is the set of pairs $(i,\delta)\in\K_g\times G$ with positive transition rate $c(i,\delta)>0$.

\begin{proof}
Inequality \eqref{4.aux} is proved in \cite{Con22} for a slightly different model. We adapt the proof to the present situation. The idea is to split the left-hand side of \eqref{4.aux} into two parts and to reformulate both parts separately:
\begin{align}\label{4.IJ}
  &\sum_{\substack{(i,\delta)\in S \\ \gamma\in G^*_i,\, 
  \bar{\gamma}\in G^*_{\delta i}}} c(i,\delta) 
  \mathbf{c}(i,\delta i, \gamma, \bar{\gamma})
  \big(f^{\phi}(\gamma i,\bar{\gamma}\delta i)
  -f^{\phi}(i,\delta i)\big) m_h(i) = I-J, \quad\mbox{where} \\
  & I := \sum_{\substack{(i,\delta)\in S \\ \gamma\in G^*_i,\, 
  \bar{\gamma}\in G^*_{\delta i}}}c(i,\delta) 
  \mathbf{c}(i,\delta i, \gamma, \bar{\gamma})
  f^{\phi}(\gamma i,\bar{\gamma}\delta i)m_h(i), \nonumber \\
  & J := \sum_{\substack{(i,\delta)\in S \\ \gamma\in G^*_i,\, 
  \bar{\gamma}\in G^*_{\delta i}}}c(i,\delta) 
  \mathbf{c}(i,\delta i, \gamma, \bar{\gamma})
  f^{\phi}(i,\delta i)m_h(i). \nonumber 
\end{align}
To simplify the notation, we write $\na_\delta c(i,\gamma) 
= c(\delta i,\gamma)-c(i,\gamma)$ in the following. 

{\em Step 1: Reformulation of $J$.} We insert the coupling rates \eqref{4.coupl} into the expression for $J$ and observe that we have always two possibilities to move ($\pm_j$ and $\pm_\ell$). This leads to 
\begin{align*}
  J &= \sum_{i\in\K_h,\,j=1,\ldots,d}c(i,+_j)A^+(i,j)
  f^\phi(i,i+he_j)m_h(i) \\
  &\phantom{xx}+ \sum_{i\in\K_h,\,j=1,\ldots,d}c(i,-_j)A^-(i,j)
  f^\phi(i,i-he_j)m_h(i), \quad\mbox{where} \\
  A^\pm(i,j) &= \sum_{\ell=1}^d\big(\min\{c(i,+_\ell),c(i\pm he_j,+_\ell)\}
  + \min\{c(i,-_\ell),c(i\pm he_j,-_\ell)\}\big) \\
  &\phantom{xx}+ \sum_{\ell=1,\,\ell\neq j}^d\big(\max\{\na_{\pm j}
  c(i,+_\ell),0\} + \max\{\na_{\pm j}c(i,-_\ell,0)\}\big) \\
  &\phantom{xx}+ \sum_{\ell=1,\,\ell\neq j}^d\big(\max\{-\na_{\pm j}
  c(i,+_\ell),0\} + \max\{-\na_{\pm j}c(i,-_\ell),0\}\big) \\
  &\phantom{xx}+ \kappa_\pm(i,j) + \kappa_\mp(i\pm he_j,j).
\end{align*}
We insert into $A^+(i,j)$ definition \eqref{3.kappa+} for $\kappa_\pm$, namely
\begin{align*}
  \kappa_+(i,j) &= c(i,+_j) - c(i+he_j,+_j) \\
  &\phantom{xx}
  - \sum_{\ell=1,\ell\neq j}^d\big(\max\{\na_{+j}c(i,+_\ell),0\}
  + \max\{\na_{+j}c(i,-_\ell),0\}\big), \\
  \kappa_-(i+he_j,j) &= c(i+he_j,-_j) - c(i,-_j) \\
  &\phantom{xx}
  - \sum_{\ell=1,\ell\neq j}^d\big(\max\{-\na_{+j}c(i,+_\ell),0\}
  + \max\{-\na_{+j}c(i,-_\ell),0\}\big),
\end{align*}
which cancels the terms in $A^+(i,j)$ involving the maximum, and use the identity $\min\{a,b\}=a-\max\{a-b,0\}$ for $a,b\in\R$ to find that
\begin{align*}
  A^+(i,j) & = \sum_{\ell=1}^d \big(\min(c(i,+_\ell),c(i+he_j,+_\ell))
  + \min(c(i,-_\ell),c(i+he_j,-_\ell))\big) \\
  &\phantom{xx}+ c(i,+_j)-c(i+he_j,+_j)
  - c(i,-_j)+c(i+he_j, -_j) \\
  &= \sum_{\ell=1}^d \big(c(i,+_\ell) 
  - \max\{0,-\nabla_{+j}c(i,+_\ell)\} 
  + c(i,-_\ell) - \max\{0,-\nabla_{+j}c(i,-_\ell)\}\big) \\ 
  &\phantom{xx}+ c(i,+_j)-c(i+he_j,+_j) - c(i,-_j)
  + c(i+he_j, -_j) \\
  &= \sum_{\ell=1}^d (c(i,+_\ell)+c(i,-_\ell))  + \kappa_-(i+he_j,j).
\end{align*}
Proceeding in a similar way, we obtain
$$
  A^-(i,j) = \sum_{\ell=1}^d\big(c(i,+_\ell) + c(i,-_\ell)\big)
  + \kappa_+(i-he_j,j).
$$
This shows that
\begin{align}
  J &= \sum_{i\in\K_h,\,j=1,\ldots,d}c(i,+_j)\bigg(\sum_{\ell=1}^d
  \big(c(i,+_\ell) + c(i,-_\ell)\big) + \kappa_-(i+he_j,j)\bigg)
  f^\phi(i,i+he_j)m_h(i) \nonumber \\
  &\phantom{xx}+ \sum_{i\in\K_h,\,j=1,\ldots,d}c(i,-_j)
  \bigg(\sum_{\ell=1}^d
  \big(c(i,+_\ell) + c(i,-_\ell)\big) + \kappa_+(i-he_j,j)\bigg)
  f^\phi(i,i-he_j)m_h(i). \label{4.J} 
\end{align}

{\em Step 2: Reformulation of $I$.}
Setting $B_{(i,\delta,\gamma,\bar\gamma)}:=c(i,\delta)\mathbf{c}(i,\delta i,\gamma,\bar\gamma)f^\phi(\gamma i,\bar\gamma\delta i)m_h(i)$, we write \begin{align*}
  I &= \sum_{\substack{(i,\delta)\in S \\ 
  \gamma,\bar\gamma\in G^*}}B_{(i,\delta,\gamma,\bar\gamma)}
  = I_1+I_2+I_3+I_4, \quad\mbox{where} \\
  I_1 &= \sum_{\substack{i\in\K_h \\ 
  j=1,\ldots,d}}\sum_{\ell=1,\,\ell\neq j}^d\big(
  B_{(i,+_j,+_\ell,+_\ell)} + B_{(i,+_j,-_\ell,-_\ell)}
  + B_{(i,-_j,+_\ell,+_\ell)} + B_{(i,-_j,-_\ell,-_\ell)}\big), \\
  I_2 &= \sum_{\substack{i\in\K_h \\ j=1,\ldots,d}}\big(
  B_{(i,+_j,+_j,+_j)} + B_{(i,-_j,+_j,+_j)}
  + B_{(i,+_j,-_j,-_j)} + B_{(i,-_j,-_j,-_j)}\big), \\
  I_3 &= \sum_{\substack{i\in\K_h \\ j=1,\ldots,d}}
  \sum_{\ell=1,\,\ell\neq j}^d\big(
  B_{(i,+_j,+_j,+_\ell)} + B_{(i,+_j,+j,-_\ell)}
  + B_{(i,+_j,+_\ell,-_j)} + B_{(i,+_j,-_\ell,-_j)}\big), \\
  I_4 &= \sum_{\substack{i\in\K_h \\ j=1,\ldots,d}}
  \sum_{\ell=1,\,\ell\neq j}^d\big(
  B_{(i,-_j,+_\ell,+_j)} + B_{(i,-_j,-_\ell,+_j)}
  + B_{(i,-_j,-_j,+_\ell)} + B_{(i,-_j,-_j,-_\ell)}\big). \\
\end{align*}
The term $I_1$ is reformulated by inserting the coupling rates and performing the shift $i\mapsto i\pm he_j$:
\begin{align*}
  I_1 &= \sum_{\substack{i\in\K_h \\ j=1,\ldots,d}}
  \sum_{\ell=1,\,\ell\neq j}^d\Big[\Big(c(i-he_\ell,+_j) \min\{c(i-he_\ell,+_\ell),c(i+he_j-he_\ell,+_\ell)\}m_h(i-he_\ell) \\ 
  &\phantom{xx}+c(i+he_\ell,+_j) \min\{c(i+he_\ell,-_\ell),c(i+he_j+he_\ell,-_\ell)\}m_j(i+he_\ell)
  \Big)f^{\phi} (i,i+he_j) \\ 
  &\phantom{xx} +\Big(c(i-he_\ell,-_j) \min\{c(i-he_\ell,+_\ell),c(i-he_j-he_\ell,+_\ell)\}m_h(i-he_\ell) \\
  &\phantom{xx} +c(i+he_\ell,-_j) \min\{c(i+he_\ell,-_\ell),c(i-he_j+he_\ell,-_\ell)\}m_h(i+he_\ell)
  \Big)f^{\phi}(i,i-he_j)\Big].
\end{align*}
Assumption (A3) implies that $c(i,\pm_j)>c(i\pm he_j,\pm_j)$. Hence,
\begin{align*}
  \min\{c(i-he_j,+_j),c(i,+_j)\} &= c(i,+_j), \\
  \min\{c(i+he_j,-_j),c(i+2he_j,-_j)\} &= c(i+he_j,-_j), \\
  \min\{c(i-he_j,+_j),c(i-2he_j,+_j)\} &= c(i-he_j,+_j), \\
  \min\{c(i+he_j,-_j),c(i,-_j)\} &= c(i,-_j).
\end{align*}
Then, inserting the coupling rates into $I_2$, we find that
\begin{align*}
  I_2 &= \sum_{\substack{i\in\mathcal{K}_h \\ j=1,\ldots, d}}
  \Big[\Big(c(i-he_j,+_j) \min\{c(i-he_j,+_j),c(i,+_j)\}m_h(i-he_j) \\ 
  &\phantom{xx} + c(i+he_j,+_j) 
  \min\{c(i+he_j,-_j),c(i+2he_j,-_j)\}m_h(i+he_j)\Big)f^\phi(i,i+he_j) \\
  &\phantom{xx} + \Big(c(i-he_j,-_j) \min\{c(i-he_j,+_j),c(i-2he_j,+_j)\}m_h(i-he_j) \\ 
  &\phantom{xx} + c(i+he_j,-_j) \min\{c(i+he_j,-_j),c(i,-_j)\}m_h(i+he_j)
  \Big)f^{\phi}(i,i-he_j)\Big] \\ 
  &= \sum_{\substack{i\in\mathcal{K}_h \\ j=1,\ldots, d}} \Big[\Big( c(i-he_j,+_j) c(i,+_j)m_h(i-he_j) 
  + c(i+he_j,+_j) c(i+he_j,-_j)m_h(i+he_j)\Big) \\
  &\phantom{xx}\times f^{\phi} (i,i+he_j) \\ 
  &\phantom{xx} + \Big(c(i-he_j,-_j) c(i-he_j,+_j)m_h(i-he_j) 
  + c(i+he_j,-_j) c(i,-_j)m_h(i+he_j)\Big) \\
  &\phantom{xx}\times f^{\phi}(i,i-he_j)\Big] \\ 
  &= \sum_{\substack{i\in\mathcal{K}_h \\ j=1,\ldots, d}} \Big[\Big( c(i-he_j,+_j) c(i,+_j)m_h(i-he_j) + c(i,+_j) c(i+he_j,-_j)m_h(i+he_j)
  \Big) \\
  &\phantom{xx}\times f^{\phi} (i,i+he_j) \\ 
  &\phantom{xx} + \Big(c(i,-_j) c(i-he_j,+_j)m_h(i-he_j)
  + c(i+he_j,-_j) c(i,-_j)m_h(i+he_j)\Big) \\
  &\phantom{xx}\times f^{\phi}(i,i-he_j) \\ 
  &\phantom{xx} - \na_{-_j}c(i,+_j) c(i,-_j)m_h(i)f^{\phi}(i,i-he_j)
  - \na_{+_j} c(i,-_j) c(i,+_j)m_h(i)f^{\phi}(i,i+he_j)\Big],
\end{align*}
where in the last step, we split the second and third terms and performed the shifts $i\mapsto i-he_j$ and $i\mapsto i+he_j$ to these terms, respectively. Similarly, we insert the coupling rates, perform the shifts $i\mapsto i\pm he_j$, and exchange the variables $j$ and $\ell$, leading to
\begin{align*}
  I_3 & = \sum_{\substack{i\in\mathcal{K}_h \\ j=1,\ldots, d}} \sum_{\substack{\ell=1 \\ \ell \neq j}}^d \Big(c(i-he_\ell,+_\ell)
  \max\{-\nabla_{-_\ell} c(i,+_j),0\}m_h(i-he_\ell)f^{\phi}(i,i+he_j) \\ 
  &\phantom{xx} + c(i-he_\ell,+_\ell) 
  \max\{-\nabla_{-_l}c(i,-_j),0\}m_h(i-he_\ell) f^{\phi}(i,i-he_j) \\ 
  &\phantom{xx} + c(i,+_\ell)\max\{-\nabla_{+_\ell}c(i,+_j),0\} 
  m_h(i) f^{\phi}(i,i+he_j) \\ 
  &\phantom{xx} + c(i,+_\ell) \max\{-\nabla_{+_\ell}c(i,-_j),0\} 
  m_h(i) f^{\phi}(i,i-he_j)\Big), \\
  I_4 &= \sum_{\substack{i\in\mathcal{K}_h \\ j=1,\ldots, d}} \sum_{\substack{\ell=1 \\ \ell \neq j}}^d \Big(c(i,-_\ell) \max\{-\nabla_{-_\ell}c(i,+_j),\})m_h(i)f^{\phi}(i,i+he_j) \\ 
  &\phantom{xx} + c(i,-_\ell) 
  \max\{-\nabla_{-_\ell}c(i,-_j),0\} m_h(i) f^{\phi}(i,i-he_j) \\ 
  &\phantom{xx} + c(i+he_\ell,-_\ell) 
  \max\{-\nabla_{+_\ell}c(i,+_j),0\} m_h(i+he_\ell) f^{\phi}(i,i+he_j) \\
  &\phantom{xx} + c(i+he_\ell,-_\ell) 
  \max\{-\nabla_{+_\ell}c(i,-_j),0\} m_h(i+he_\ell) f^{\phi}(i,i-he_j) 
  \Big).
\end{align*}
We collect first all terms with factor $f^\phi(i,i+he_j)$ appearing in $I_1$, $I_3$, and $I_4$ and reformulate them by observing that, by \eqref{2.path}, we have
\begin{align*}
  c(i,+_j)\max\{-\na_{+j}c(i,\pm_\ell),0\} 
  = c(i,\pm_\ell)\max\{-\na_{\pm\ell}c(i,+_j),0\}.
\end{align*}
This yields
\begin{align}\label{4.fplus}
  & \sum_{\substack{i\in\mathcal{K}_h \\ j=1,\ldots, d}} \sum_{\substack{\ell=1 \\ \ell \neq j}}^d \Big(c(i-he_\ell,+_j)\min\{c(i-he_\ell,+_\ell),
  c(i+he_j-he_\ell,+_\ell)\}m_h(i-he_\ell) \\
  &\phantom{xx}+ c(i+he_\ell,+_j)
  \min\{c(i+he_\ell,-_\ell), c(i+he_j+he_\ell,-_\ell)\}m_h(i+he_\ell) 
  \nonumber \\
  &\phantom{xx} + c(i-he_\ell,+_\ell)
  \max\{-\nabla_{-_l}c(i,+_j),0\} m(i-he_\ell) \nonumber \\
  &\phantom{xx} + c(i+he_\ell,-_\ell)
  \max\{-\nabla_{+_\ell}c(i,+_j),0\} m_h(i+he_\ell) \nonumber \\
  &\phantom{xx} + c(i,+_\ell)\max\{-\nabla_{+_\ell}c(i,+_j),0\} m_h(i)
  + c(i,-_\ell)\max\{-\nabla_{-_\ell}c(i,+_j),0\}m_h(i)
  \Big) \nonumber \\
  &\phantom{xx}\times f^{\phi}(i,i+he_j) \nonumber \\ 
  &= \sum_{\substack{i\in\mathcal{K}_h \\ j=1,\ldots, d}} \sum_{\substack{\ell=1 \\ \ell \neq j}}^d \Big( 
  c(i,+_j) c(i-he_\ell,+_\ell)m_h(i-he_\ell) 
  + c(i+he_\ell,-_\ell)c(i,+_j)m_h(i+he_\ell) \nonumber \\ 
  &\phantom{xx} + c(i,+_j)\max\{-\nabla_{+_j}c(i,+_\ell),\} m_h(i)
  + c(i,+_j)\max\{-\nabla_{+_j}c(i,-_\ell),0\} m_h(i)
  \Big) \nonumber \\
  &\phantom{xx}\times f^{\phi}(i,i+he_j). \nonumber 
\end{align}
Similarly, by collecting the terms with factor $f^\phi(i,i-he_j)$ in $I_1$, $I_3$, and $I_4$, we arrive at
\begin{align}\label{4.fminus}
  &\sum_{\substack{i\in\mathcal{K}_h \\ j=1,\ldots, d}} \sum_{\substack{\ell=1 \\ \ell \neq j}}^d \Big( 
  c(i,-_j) c(i-he_\ell,+_\ell)m_h(i-he_\ell) 
  + c(i+he_\ell,-_\ell)c(i,-_j)m_h(i+he_\ell) \\ 
  &\phantom{xx} + c(i,-_j)\max\{-\nabla_{-_j}c(i,+_\ell),0\} m_h(i)
  + c(i,-_j)\max\{-\nabla_{-_j}c(i,-_\ell),0\} m_h(i) \Big) \nonumber \\ &\phantom{xx}\times f^{\phi}(i,i-he_j)). \nonumber 
\end{align}
We insert $I_2$ and \eqref{4.fplus}--\eqref{4.fminus} into $I$, giving
\begin{align*}
  I &= \sum_{\substack{i\in\mathcal{K}_h \\ j=1,\ldots, d}}
  \big(c(i,+_j)f^\phi(i,i+he_j)E^+ +  c(i,-_j)f^\phi(i,i-he_j)E^-\big),
  \quad\mbox{where} \\
  E^\pm &= \sum_{\ell=1,\,\ell\neq j}^d
  \big[c(i-he_\ell,+_\ell)m_h(i-he_\ell) 
  + c(i+he_\ell,-_\ell)m(i+he_\ell)\big] \\
  &\phantom{xx}+ c(i-he_j,+_j)m_h(i-he_j) + c(i+he_j,-_j)m_h(i+he_j) 
  - \na_{\pm j}c(i,\mp_j)m_h(i) \\ 
  &\phantom{xx}
  + \sum_{\ell=1, \ell \neq j}^d\big[\max\{-\nabla_{\pm_j}c(i,+_\ell),0\} 
  + \max\{-\nabla_{\pm_j}c(i,-_\ell),0\}\big]m_h(i).
\end{align*}
We combine the sum over $\ell\neq j$ and the second and third terms in $E^+$ to obtain
\begin{align*}
  E^+ &= \sum_{\ell=1}^d
  \big[c(i-he_\ell,+_\ell)m_h(i-he_\ell) 
  + c(i+he_\ell,-_\ell)m(i+he_\ell)\big] - \na_{+j}c(i,-_j)m_h(i) \\
  &\phantom{xx}
  + \sum_{\ell=1, \ell \neq j}^d\big[\max\{-\nabla_{+_j}c(i,+_\ell),0\} 
  + \max\{-\nabla_{+_j}c(i,-_\ell),0\}\big]m_h(i) \\
  &= \sum_{\ell=1}^d\big(c(i,+_\ell) + c(i,-_\ell)\big)
  - \kappa_-(i+ he_j,j),
\end{align*}
where the last step follows from definition \eqref{3.kappa+} of $\kappa_+$ and the identity
\begin{align*}
  \sum_{\ell=1}^d&\big(c(i-he_\ell,+_\ell)m_h(i-he_\ell) + c(i+he_\ell,-_\ell)m_h(i+he_\ell)\big) \\
  &= \sum_{\ell=1}^d (c(i,+_\ell)+c(i,-_\ell))m_h(i)
  \quad\mbox{for }i\in\K_h,
\end{align*}
which is similarly shown as in the proof of Lemma \ref{lem.invmeas}. A similar computation is done for $E^-$, resulting eventually in 
\begin{align*}
  E^\pm = \sum_{\ell=1}^d\big(c(i,+_\ell) + c(i,-_\ell)\big)
  - \kappa_\mp(i\pm he_j,j).
\end{align*}
This shows that
\begin{align}\label{4.I}
  I &= \sum_{\substack{i\in\mathcal{K}_h \\ j=1,\ldots, d}}
  \bigg[c(i,+_j)f^\phi(i,i+he_j)\bigg(\sum_{\ell=1}^d\big(c(i,+_\ell) 
  + c(i,-_\ell)\big) - \kappa_-(i+he_j,j)\bigg) \\
  &\phantom{xx}
  + c(i,-_j)f^\phi(i,i-he_j)\Big(\sum_{\ell=1}^d\big(c(i,+_\ell) 
  + c(i,-_\ell)\big) - \kappa_+(i - he_j,j)\bigg)\bigg]m_h(i). \nonumber 
\end{align}

{\em Step 3: End of the proof.} We compute the difference $I-J$ by inserting the expressions derived in Steps 1 and 2 and observing that the terms involving $c(i,\pm_\ell)$ cancel out:
\begin{align*}
  I-J &= \sum_{\substack{i\in\mathcal{K}_h \\ j=1,\ldots, d}}  c(i,+_j)f^{\phi}(i,i+he_j) (-2\kappa_-(i+he_j,j)) m_h(i) \\ 
  &\phantom{xx} + \sum_{\substack{i\in\mathcal{K}_h \\ j=1,\ldots, d}}  c(i,-_j)f^{\phi}(i,i-he_j) (-2\kappa_+(i-he_j,j)) m_h(i) \\
  &\le -\kappa_\phi\sum_{(i,\delta)\in S}c(i,\delta)
  f^\phi(i,\delta i)m_h(i),
\end{align*}
with $\kappa_\phi=\min_{i, i+he_j\in\K_h,\,j=1,\ldots,d}
\{\kappa_-(i+he_j,j)+\kappa_+(i,j)\}$. In the last step,
we applied 
\begin{align*}
  c(i,\pm_j)&f^{\phi}(i,i\pm he_j) (-\kappa_\mp(i \pm he_j,j)) m_h(i) \\ 
  & =c(i\mp he_j,\pm_j)f^{\phi}(i\mp he_j,i) (-\kappa_\mp(i,j))
   m_h(i\mp he_j) \\ 
  & =c(i,\mp_j)f^{\phi}(i,i\mp he_j) (-\kappa_\mp (i,j)) m_h(i),
\end{align*}
using the symmetry of $f^\phi$, $f^\phi(i,i\pm he_j)\ge 0$, and $c(i\pm he_j, \mp j)m(i\pm he_j)=c(i,\pm j)m(i)$. Note that the terms involving $\kappa_-(i+he_j,j)$ for $i+he_j\notin \K_h$ and $\kappa_+(i-he_j,j)$ for $i-he_j\notin \K_h$ vanish since the rates $c(i,+_j)$ for $i+he_j\notin \K_h$ and $c(i,-_j)$ for $i-he_j\notin \K_h$ are zero. Inserting this inequality into \eqref{4.IJ} 
 finishes the proof.
\end{proof}

\begin{lemma}[Conditions (ii) and (iii)]
Let Assumptions (A2)--(A3) hold. Then conditions {\rm (ii)} and {\rm (iii)} of Theorem \ref{thm.conforti} are satisfied with $2\kappa''=\kappa'''=\kappa_\phi$.
\end{lemma}

\begin{proof}
By definition \eqref{4.coupl} of the coupling rates, we have
\begin{align*}
  &\inf_{(i,\delta)\in S}\min\{\mathbf{c}(i,\delta i,\delta,e),
  \mathbf{c}(i,\delta i,e,\delta^{-1})\} \\
  &\phantom{xxxx}= \inf_{\substack{i\in\K_h,\,j=1,\ldots,d \\ i+he_j\in\K_h}}
  \min\{\mathbf{c}(i,i+he_j,+_j,e),
  \mathbf{c}(i,i+he_j,e,-_j)\} \\
  &\phantom{xxxx}= \inf_{\substack{i\in\K_h,\,j=1,\ldots,d \\ i+he_j\in\K_h}}
  \min\{\kappa_+(i,j),\kappa_-(i+he_j,j)\} = \frac{\kappa_\phi}{2}, \\
  &\inf_{(i,\delta)\in S}\sum_{\substack{\gamma\in G^*_i,\, 
  \bar{\gamma}\in G^*_{\delta i} \\ 
  \gamma i=\bar{\gamma}\delta i}}
  \mathbf{c}(i,\delta i, \gamma,\bar{\gamma})
  = \inf_{\substack{i\in\mathcal{K}_h, j=1,\ldots,d: \\ i+he_j\in\mathcal{K}_h }} (\kappa_+(i,j)+\kappa_-(i,j))
  \ge \kappa_{\phi},
\end{align*}
since $\kappa_\pm(i,j)\ge \kappa_\phi/2$.
\end{proof}

Theorem \ref{thm.decay} now follows directly from Theorem \ref{thm.conforti}.


\subsection{Exponential decay in Wasserstein distance}

We prove Theorems \ref{thm.W} and \ref{thm.W1}.

{\em Step 1: Decay in the $L^2$ Wasserstein distance.} Because of the infimum property of the $L^2$ Wasserstein distance, it is sufficient to find a coupling of two copies $(Y_t^1)_{t\ge 0}$ and $(Y_t^2)_{t\ge 0}$ of the Markov chain (with initial data $\nu$ and $\eta$, respectively) such that 
$$
  \frac{\dd}{\dd t}\E(|Y_t^1-Y_t^2|^2) 
  \le -\kappa\E(|Y_t^1-Y_t^2|^2) + O(h),
$$
since by Gr\"onwall's inequality and the property $\mathcal{W}_2(\nu p_t,\eta p_t)^2\le \E(|Y_t^1-Y_t^2|^2)$, the conclusion follows after taking the infimum over all couplings of $\nu$ and $\eta$ and taking the square root. 

Consider the transition rates \eqref{1.c}. We construct a coupling of two copies of the transition kernel by coupling the transition rates in the following way. For $i,k\in\K_h$ and $\gamma,\bar\gamma\in G$, we set
\begin{align}\label{5.coupl}
  \mathbf{c}(i,k,\gamma,\bar\gamma)
  = \begin{cases}
  \min\{c(i,\gamma),c(k,\bar\gamma)\} 
  &\mbox{if }\gamma=\bar\gamma\in G, \\
  \max\{c(i,\gamma)-c(k,\bar\gamma),0\}
  &\mbox{if }\gamma\in G,\,\bar\gamma=e, \\
  \max\{c(i,\gamma)-c(k,\bar\gamma),0\}
  &\mbox{if }\bar\gamma\in G,\,\gamma=e, \\
  0 &\mbox{else}.
  \end{cases}
\end{align}
By definition \eqref{4.coupl}, this defines indeed a coupling. We denote the generator for the corresponding Markov chain on $\K_h\times\K_h$ by $\L_h^2$. We infer from the definition of the coupling that
\begin{align*}
  \L_h^2|i-k|^2 &= \sum_{\gamma,\,\bar\gamma\in G^*}
  \mathbf{c}(i,k,\gamma,\bar\gamma)\big(
  |\gamma i-\bar\gamma k|^2 - |i-k|^2\big) \\
  &= \sum_{\gamma\in G}\big[\max\{c(i,\gamma)-c(k,\gamma),0\}
  \big(|\gamma i-k|^2 - |i-k|^2\big) \\
  &\phantom{xx} + \max\{c(k,\gamma)-c(i,\gamma),0\}
  \big(|i-\gamma k|^2 - |i-k|^2\big)\big].
\end{align*}
As $G$ consists of the moves $\pm_j$ with $j=1,\ldots,d$, this sum becomes
\begin{align}\label{5.Lh2}
  & \L_h^2|i-k|^2 = \sum_{j=1}^d\big[C^+\big(|i+he_j-k|^2-|i-k|^2\big)
  + C^-\big(|i-he_j-k|^2-|i-k|^2\big)\big], \\
  & \mbox{where }C^\pm = \max\{c(i,\pm_j)-c(k,\pm_j),0\} 
  + \max\{c(k,\mp_j)-c(i,\mp_j),0\}. \nonumber 
\end{align}
Since we have assumed that the potential is additive, $V(i)=\sum_{j=1}^d V_j(i_j)$, we have $V^h(i+he_j)-V^h(i)=V_j^h(i_j+h)-V_j^h(i_j)$. Then it follows from definition \eqref{1.c} of $c(i,\pm_j)$ that
\begin{align*}
  c(i,\gamma)-c(k,\gamma) 
  &= \frac{\sigma^2}{h^2}\big(e^{-(V_j^h(i_j+h)-V_j^h(i_j))
  /(2\sigma^2)} - e^{-(V_j^h(k_j+h)-V_j^h(k_j))/(2\sigma^2)}\big) \\
  &= \frac{\sigma^2}{h^2}\exp\bigg(-\frac{1}{2\sigma^2}\int_0^h
  \pa V_j^h(i_j+s)\dd s\bigg) \\
  &\phantom{xx}\times\bigg[1 - \exp\bigg(
  -\frac{1}{2\sigma^2}\int_0^h\big(\pa V_j^h(k_j+s)-\pa V_j^h(i_j+s)
  \big)\dd s\bigg)\bigg],
\end{align*}
where $\pa V_j^h$ denotes the derivative of the scalar function $V_j^h$, interpreted as a function defined on $D$. This gives
\begin{align*}
  C^+ &= \frac{\sigma^2}{h^2}\max\bigg\{
  \exp\bigg(-\frac{1}{2\sigma^2}(V_j^h(i_j+h)-V_j^h(i_j))\bigg) \\
  &\phantom{xx}\times\bigg[1 - \exp\bigg(
  -\frac{1}{2\sigma^2}\int_0^h\big(\pa V_j^h(k_j+s)-\pa V_j^h(i_j+s)
  \big)\dd s\bigg)\bigg],0\bigg\} \\
  &\phantom{xx}+ \frac{\sigma^2}{h^2}\max\bigg\{
  \exp\bigg(-\frac{1}{2\sigma^2}(V_j^h(i_j-h)-V_j^h(i_j))\bigg) \\
  &\phantom{xx}\times\bigg[\exp\bigg(
  \frac{1}{2\sigma^2}\int_0^h\big(\pa V_j^h(k_j-s)-\pa V_j^h(i_j-s)
  \big)\dd s\bigg)-1\bigg],0\bigg\}.
\end{align*}
Since $\max\{a,0\}=a\mathrm{1}_{\{a>0\}}$, we may remove the maximum by introducing the factor $\mathrm{1}_{\{k_j>i_j\}}$. Then we deduce from Assumption (A2) that
$$
  \pa V_j^h(k_j+s)-\pa V_j^h(i_j+s) \ge \kappa(k_j-i_j)
  \quad\mbox{on }\{k_j>i_j\},
$$
and hence,
\begin{align*}
  C^+ &\ge \frac{\sigma^2}{h^2}
  \exp\bigg(-\frac{1}{2\sigma^2}(V_j^h(i_j+h)-V_j^h(i_j))\bigg)
  \bigg[1 - \exp\bigg(-\frac{h\kappa}{2\sigma^2}(k_j-i_j)\bigg)\bigg]
  \mathrm{1}_{\{k_j>i_j\}} \\
  &\phantom{xx}+ \frac{\sigma^2}{h^2}
  \exp\bigg(-\frac{1}{2\sigma^2}(V_j^h(i_j-h)-V_j^h(i_j))\bigg)
  \bigg[\exp\bigg(\frac{h\kappa}{2\sigma^2}(k_j-i_j)\bigg)-1\bigg]
  \mathrm{1}_{\{k_j>i_j\}}.
\end{align*}
By assumption, $\na V$ is Lipschitz continuous, and so is $\pa V_j^h$.  Then the Taylor expansions $1-\exp(-x)=x+O(x^2)$ and $\exp(x)-1=x+O(x^2)$ show that
\begin{align*}
  C^+ &\ge \frac{\sigma^2}{h^2}\bigg(\frac{h\kappa}{\sigma^2}(k_j-i_j)
  - O(h^2(k_j-i_j)^2)\bigg)\mathrm{1}_{\{k_j>i_j\}}
  = \frac{1}{h}(\kappa-O(h))(k_j-i_j)\mathrm{1}_{\{k_j>i_j\}}.
\end{align*}
A similar computation leads to
$C^- \ge h^{-1}(\kappa-O(h))(i_j-k_j)\mathrm{1}_{\{k_j<i_j\}}$.
We insert these estimations into \eqref{5.Lh2} and use
$|i\pm he_j-k|^2-|i-k|^2=-2h|i_j-k_j|+h^2$ (which is negative for sufficiently small $h>0$):
\begin{align*}
  \L_h^2|i-k|^2 &\le (\kappa-O(h))\sum_{j=1}^d
  |k_j-i_j|(\mathrm{1}_{\{k_j>i_j\}} + \mathrm{1}_{\{k_j<i_j\}})
  (-2|i_j-k_j|+h) \\
  &= -2\kappa|i-k|^2 + O(h).
\end{align*}

We infer for the two copies $(Y_t^1)$ and $(Y_t^2)$ of the Markov chain that
\begin{align*}
  \frac{\dd}{\dd t}\E(|Y_t^1-Y_t^2|^2)
  = \E\big(\L_h^2|Y_t^1-Y_t^2|^2\big) 
  \le -2\kappa\E(|Y_t^1-Y_t^2|^2) + O(h),
\end{align*}
and the result follows after an application of Gr\"onwall's lemma, as detailed above.

{\em Step 2: Decay in the $L^p$ Wasserstein distance.} Let $1<p<\infty$ with $p\neq 2$. The proof is similar as in Step 1, but we need the condition $d=1$. Using the coupling \eqref{5.coupl}, we obtain
\begin{align*}
  \L_h^2|i-k|^p &= \big(\mathbf{c}(i,k,+,e) + \mathbf{c}(i,k,e,-)\big)
  \big(|i+h-k|^p - |i-k|^p\big) \\
  &\phantom{xx}+ \big(\mathbf{c}(i,k,e,+) + \mathbf{c}(i,k,-,e)\big)
  \big(|i-h-k|^p - |i-k|^p\big).
\end{align*}
The assumption $d=1$ enters in the estimate
$$
  |i\pm h-k|^p - |i-k|^p = \pm p(i-k)^{p-1}h + O(h^2),
$$
which generally does not hold for $d>1$. Then, after a similar computation as in Step 1,
$$
  \L_h^2|i-k|^p \le -p\kappa|i-k|^p + O(h),
$$
and we conclude with Gr\"onwall's inequality. This finishes the proof of Theorem \ref{thm.W}. 

{\em Step 3: Decay in the $L^1$ Wasserstein distance.} We turn to the proof of Theorem \ref{thm.W1}. Consider the graph distance $\mathrm{d}(i,k)=\sum_{j=1}^d|i_j-k_j|$ for $i,k\in\K_h$. With respect to this distance, the set $\mathcal{K}_h$ forms a geodesic graph in the sense that for every $i,k\in\mathcal{K}_h$, there exists a path $i=i^{(0)},i^{(1)},\ldots,i^{(n)}=k$ such that $\mathrm{d}(i,k)=\sum_{\ell=1}^n\mathrm{d}(i^{(\ell)},i^{(\ell-1)})$. We remark that an edge between $i^{(\ell)}$ and $i^{(\ell+1)}$ exists in the graph if there exists $\gamma\in G$ such that $\gamma i^{(\ell)}=i^{(\ell+1)}$ (and vice versa, there exists $\bar{\gamma}\in G$ such that $\bar{\gamma}i^{(\ell+1)}=i^{(\ell)}$).

By the path coupling method of \cite{BuDy97}, it is sufficient for obtaining contraction in Wasserstein distance with respect to the distance d to prove contraction with respect to the distance d for every neighbouring states $(i,k)$. Indeed, let $(i,k)$ be two neighbouring states, i.e., there exists $\gamma \in G$ such that $\gamma i=k$ (and vice versa, there exists $\bar{\gamma}\in G$ such that $\bar{\gamma}k=i$). Note that $\mathrm{d}(i,k)=h$. Without loss of generality, we can assume that $i,k\in\mathcal{K}_h$ is such that $k=i+he_j$ for some $j\in\{1,\ldots, d\}$. Recalling the coupling rates \eqref{4.coupl}, it holds for the corresponding generator $\mathcal{L}^2_h$ on the product space that
\begin{align*}
  \mathcal{L}^2_h\mathrm{d}(i,k)
  &= \sum_{\gamma,\, \bar{\gamma}\in G} \mathbf{c}(i,k,\gamma,\bar{\gamma})\big(\mathrm{d}(\gamma i,\bar{\gamma}k)-\mathrm{d}(i,k)\big) \\ 
  &= \kappa_+(i,j)\big(\mathrm{d}(i+he_j,k)
  - \mathrm{d}(i,k)\big) 
  + \kappa_-(k,j)\big(\mathrm{d}(i,k-he_j)-\mathrm{d}(i,k)\big) \\
  &= -\big(\kappa_+(i,j)+\kappa_-(i+he_j,j)\big)\mathrm{d}(i,k)
  \le -\kappa_1\mathrm{d}(i,k),
\end{align*}
where $\kappa_1= \min_{j=1,\ldots,d,\,i,i+he_j\in\mathcal{K}_h} (\kappa_+(i,j)+\kappa_-(i+he_j,j))$.
It follows from the path coupling method that for two continuous-time Markov chains driven by the transition rates $\mathbf{c}$ given in \eqref{4.coupl},
\begin{align*}
  \frac{\dd}{\dd t}\mathbb{E}[\mathrm{d}(Y_t^1,Y_t^2)]
  \le -\kappa_1 \mathbb{E}[\mathrm{d}(Y_t^1,Y_t^2)],
\end{align*}
and hence by Gr\"onwall's inequality, 
\begin{align*}
  \mathcal{W}_{\dd,1}(\nu p_t,\eta p_t )
  \le \mathbb{E}[\mathrm{d}(Y_t^1,Y_t^2)]
  \leq e^{-\kappa_1 t} \mathbb{E}[\mathrm{d}(Y_0^1, Y_0^2)],
\end{align*}
recalling that $\mathcal{W}_{\dd, 1}$ denotes the $L^1$ Wasserstein distance with respect to the graph distance d. We take the infimum over all couplings $\nu$ and $\eta$:
\begin{align*}
  \mathcal{W}_{\dd, 1}(\nu p_t,\eta p_t)
  \leq e^{-\kappa_1 t}\mathcal{W}_{\dd, 1}(\nu,\eta).
\end{align*}
We deduce from the equivalence of the distance d and the Euclidean distance, $|x-y|\le \mathrm{d}(x,y)\le \sqrt{d}|x-y|$, that
\begin{align*}
  \mathcal{W}_{1}(\nu p_t, \eta p_t)
  \leq \sqrt{d} e^{-\kappa_1 t}\mathcal{W}_{1}(\nu, \eta ),
\end{align*}
which concludes the proof.


\section{Proofs for discrete-time Markov chains}\label{sec.proofs2}

\subsection{Exponential decay in $\phi$-entropy}

We present first the proof of Proposition \ref{prop.BE} and then of Theorem \ref{thm.ddecay}. We consider a discrete-time Markov chain with transition kernel $\pi$ (see \eqref{1.tk}), transition rates $p(i,\gamma)$ (see \eqref{3.p}), and the corresponding invariant measure $m_h$ (see Lemma \ref{lem.invmeas}).

{\em Step 1: Proof of Proposition \ref{prop.BE}.}
We show Proposition \ref{prop.BE} by following the lines of the proof of \cite[Prop.~1]{JuSc17}. It follows from assumptions (i) and (ii) that
\begin{align*}
  \mathcal{F}(\pi^{n+1}f)-\mathcal{F}(\pi^n f)
  &\le -\tau\lambda\mathcal{F}(\pi^n f) 
  \le -\tau\lambda C_P^{-1}\mathcal{P}(\pi^n f) \\
  &= \lambda C_P^{-1}\big(\mathcal{H}^\phi(\pi^{n+1}f|m_h)
  - \mathcal{H}^\phi(\pi^n f|m_h)\big).
\end{align*}
Iterating this argument leads for $k>n$ to
\begin{equation}\label{5.auxF}
  \mathcal{F}(\pi^{k}f)-\mathcal{F}(\pi^n f)
  \le \lambda C_P^{-1}\big(\mathcal{H}^\phi(\pi^{k}f|m_h)
  - \mathcal{H}^\phi(\pi^n f|m_h)\big).
\end{equation}
We infer from (ii) that 
$\mathcal{F}(\pi^{n}f)\le(1-\tau\lambda)^n\mathcal{F}(f)$ and thus 
$\lim_{k\to\infty}\mathcal{F}(\pi^k f)=0$. We perform the limit $k\to\infty$ in \eqref{5.auxF}, using (iii):
$$
  \mathcal{F}(\pi^n f)\ge\lambda C_P^{-1}\mathcal{H}^\phi(\pi^n f|m_h).
$$
Then the result follows from
\begin{align*}
  \mathcal{H}^\phi(\pi^n f|m_h) 
  &\le \lambda^{-1}C_P\mathcal{F}(\pi^n f)
  \le \lambda^{-1}C_P(1-\tau\lambda)^n\mathcal{F}(f) \\
  &= \frac{C_P\mathcal{F}(f)}{\lambda\mathcal{H}^\phi(f|m_h)}
  (1-\tau\lambda)^n\mathcal{H}^\phi(f|m_h)
  \le \frac{C_P\mathcal{F}(f)}{\lambda\mathcal{H}^\phi(f|m_h)}
  e^{-\lambda n\tau}\mathcal{H}^\phi(f|m_h).
\end{align*}

{\em Step 2: Proof of Theorem \ref{thm.ddecay}.} We need to verify conditions (i)--(iii) of Proposition \ref{prop.BE}. First, we compute
\begin{align*}
  \tau\mathcal{P}(f) 
  &= \mathcal{H}^\phi(f|m_h) - \mathcal{H}^\phi(\pi f|m_h) \\
  &= \sum_{i\in\K_h}\phi(f(i))m_h(i) 
  - \phi\bigg(\sum_{i\in\K_h}f(i)m_h(i)\bigg) \\
  &\phantom{xx}- \sum_{i\in\K_h}\phi(\pi f(i))m_h(i) 
  - \phi\bigg(\sum_{i\in\K_h}\pi f(i)m_h(i)\bigg) \\
  &= \sum_{i\in\K_h}\big(\phi(f(i))-\phi(\pi f(i))\big)m_h(i),
\end{align*}
where the last step follows from the invariance property \eqref{2.inv}. 
Jensen's inequality (recall that $\phi$ is convex) implies that
\begin{align*}
  \tau\mathcal{P}(f) 
  \ge \sum_{i\in\K_h}\big(\phi(f(i))-\pi\phi( f(i))\big)m_h(i)
  = \sum_{i\in\K_h}\big(\phi(f(i))-\phi(f(i))\big)m_h(i) = 0.
\end{align*}
To prove that $\mathcal{P}(f)\le C_P\mathcal{F}(f)$ for some $C_P>0$, we use the definition of $\phi_\alpha$, for all $x,y\ge 0$,
\begin{align}\label{5.phia}
  \phi_\alpha(x)-\phi_\alpha(y) = \begin{cases}
  \frac{1}{\alpha}(x\phi'_\alpha(x)-y\phi'_\alpha(y))
  - \frac{\alpha}{\alpha-1}(x-y) &\mbox{if }1<\alpha\le 2, \\
  (x\phi'_\alpha(x)-y\phi'_\alpha(y))
  - (x-y) &\mbox{if }\alpha=1.
  \end{cases}
\end{align}
We deduce from the definition of the Dirichlet form and the invariance property \eqref{2.inv} that, for functions $f\ge 0$,
\begin{align*}
  \mathcal{F}(f) &= \frac12\sum_{i\in\K_h}\bigg(\sum_{\gamma\in G}
  p(i,\gamma)\big(f(\gamma i)-f(i)\big)\big(\phi'_\alpha(f(\gamma i))
  - \phi'_\alpha(f(i))\big)\bigg)m_h(i) \\
  &= \frac12\sum_{i\in\K_h}\bigg(\sum_{\gamma\in G}p(i,\gamma)
  \big(2f(i)\phi'_\alpha(f(i)) - f(i)\phi'_\alpha(f(\gamma i))
  - f(\gamma i)\phi'_\alpha(f(i))\big)\bigg)m_h(i) \\
  &\ge \sum_{i\in\K_h}\bigg(f(i)\phi'_\alpha(f(i)) 
  - \frac12f(i)\phi'_\alpha(\pi f(i)) 
  - \frac12\pi f(i)\phi'_\alpha(f(i))\bigg)m_h(i),
\end{align*}
where we used Jensen's inequality for the concave function $\phi'_\alpha$ in the last step. We use definition \eqref{5.phia} of $\phi_\alpha$ and invariance property \eqref{2.inv} again to find that
\begin{align*}
  \tau\mathcal{P}(f) &= \sum_{i\in\K_h}\big(\phi_\alpha(f(i))
  - \phi_\alpha(\pi f(i))\big)m_h(i) \\
  &= \frac{1}{\alpha}\sum_{i\in\K_h}\big(f(i)\phi'_\alpha(f(i))
  - \pi f(i)\phi'_\alpha(\pi f(i))\big)m_h(i).
\end{align*}
This shows that
\begin{align*}
  \mathcal{F}(f)-\frac{\alpha\tau}{2}\mathcal{P}(f)
  \ge \frac12\sum_{i\in\K_h}\big(f(i)-\pi f(i)\big)
  \big(\phi'_\alpha(f(i))-\phi'_\alpha(\pi f(i))\big)m_h(i) \ge 0,
\end{align*}
since $(a-b)(\phi'_\alpha(a)-\phi'_\alpha(b))\ge 0$ for all $a,b\ge 0$ by the convexity of $\phi_\alpha$. This verifies condition (i) with constant $C_P=2/(\alpha\tau)$. 

To show condition (ii), we consider the same coupling \eqref{4.coupl} as for continuous-time Markov chains, except that the rates $c$ are replaced by the transition rates \eqref{3.p}. In particular, the resulting coupling rates $\mathbf{p}(i,\bar\imath,\cdot,\cdot)$ satisfy
$$
  \sum_{\gamma\in G_i^*}\sum_{\bar\gamma\in G_{\bar\imath}^*}
  \mathbf{p}(i,\bar\imath,\gamma,\bar\gamma) 
  = \sum_{\gamma\in G_i^*}p(i,\gamma)
  = \sum_{\bar\gamma\in G_{\bar\imath}^*}p(\bar\imath,\bar\gamma) = 1
  \quad\mbox{for }i,\bar\imath\in\K_h.
$$
We deduce from Lemma \ref{lem.ineq} that
\begin{align*}
  \mathcal{F}(\pi f) - \mathcal{F}(f)
  &= \frac12\sum_{\substack{(i,\delta)\in S \\ \gamma\in G_i^*,\,
  \bar\gamma\in G_{\bar\imath}^*}}p(i,\delta)
  \mathbf{p}(i,\delta i,\gamma,\bar\gamma)
  \big(f^{\phi_\alpha}(\gamma i,\bar\gamma\delta i) 
  - f^{\phi_\alpha}(i,\delta i)\big)m_h(i) \\
  &\le -\frac{\kappa_{\phi}}{2\mathcal{T}}
  \sum_{(i,\delta)\in S}p(i,\delta)
  f^{\phi_\alpha}(i,\delta i)m_h(i) 
  = -\frac{\kappa_{\phi}}{\mathcal{T}}\mathcal{F}(f).
\end{align*}
We infer that condition (ii) holds with $\tau=\mathcal{T}^{-1}$. 

Finally, we verify condition (iii). We recall that since the Markov chain $(Z_n^h)_{n\ge 0}$ is time-homogeneous, aperiodic, irreducible and defined on a finite state space, its laws converge to the unique invariant measure \cite[Theorem 1.8.3]{Nor97}, and it holds that $\pi^n f(i) \to M:= \sum_{k\in\K_h}f(k)m_h(k)$ as $n\to\infty$ for all $f\ge 0$ and $i\in\K_h$. Therefore, since $\sum_{i\in\K_h}m_h(i)=1$,
\begin{align*}
  \mathcal{H}^{\phi_\alpha}(\pi^n f|m_h) 
  &= \sum_{i\in\K_h}\phi_\alpha(\pi^n f(i))m_h(i)
  - \phi_\alpha\bigg(\sum_{i\in\K_h}\pi^n f(i)m_h(i)\bigg) \\
  &\to \sum_{i\in\K_h}\phi_\alpha(M)m_h(i) 
  - \phi_\alpha\bigg(\sum_{i\in\K_h} Mm_h(i)\bigg) = 0.
\end{align*}
Hence, we can apply Proposition \ref{prop.BE} to conclude that
$$
  \mathcal{H}^{\phi_\alpha}(\pi^n f|m_h) \le C_fe^{-\kappa_{\phi} n\tau}
  \mathcal{H}^{\phi_\alpha}(f|m_h),
$$
which finishes the proof. 


\begin{appendix}
\section{Convergence of the Markov chain to the SDE}\label{sec.conv}

We prove Theorem \ref{thm.conv}. The proof relies on the diffusion approximation of \cite[Theorem 7.4.1]{EtKu86}, and it is rather standard. We present it for completeness. For $n\in\N$, let $h=2^{-n}\xi$, where $\xi>0$ is such that $2K/h\in\N$ holds. We write $(Y_t^n)$ instead of $(Y_t^h)$, $\L_n$ instead of $\L_h$, $\K_n$ instead of $\K_h$, and $V^n$ instead of $V^h$. For each $f\in\mathcal{A}$ with $\mathcal{A}$ defined in \eqref{1.A}, we introduce the process
$$
  B_f^n(t) = f(Y_0^n) + \int_0^t \L_n(f(Y_s^n))\dd s,
$$
and we define for $f,g\in\mathcal{A}$ the process
$$
  A_{f,g}^n(t) = \int_0^t\sum_{\gamma\in G}c(Y_s^n,\gamma)
  \big(f(\gamma Y_s^n)-f(Y_s^n)\big)\big(g(\gamma Y_s^n)-g(Y_s^n)\big)
  \dd s.
$$
We notice that $(A_{f,g}^n)$ is symmetric in the sense $A_{f,g}^n=A_{g,f}^n$ and that $A_{f,f}^n(t)-A_{f,f}^n(s)$ is positive for $t>s\ge 0$. We set $\mathcal{F}_t^n=\sigma(Y_s^n,B_f^n(s),A_{f,g}^n(s):s\le t)$. The process $(Y_t^n)_{t\ge 0}$ does not explode, since it is restricted to $\K_n$. Then $M_f^n(t):=f(Y_t^n)-B_f^n(t)$ is an $\mathcal{F}_t^n$-martingale and $M_f^nM_g^n-A_{f,g}^n$ is an $\mathcal{F}_t^n$-martingale, since, using Definition \eqref{2.Lh} of $\L_n$ and denoting by $\langle\cdot,\cdot\rangle_t$ the angle-bracket processs,
\begin{align*}
  \langle &M_f^n,M_g^n\rangle_t = \int_0^t\big(
  \L_n(f(Y_s^n)g(Y_s^n)) - f(Y_s^n)\L_n(g(Y_s^n)) 
  - g(Y_s^n)\L_n(f(Y_s^n))\big)\dd s \\
  &= \int_0^t\bigg(\sum_{\gamma\in G}c(Y_s^n,\gamma)\big(f(\gamma Y_s^n)
  g(\gamma Y_s^n)-f(Y_s^n)g(Y_s^n)\big) \\
  &\phantom{xx}- f(Y_s^n)\sum_{\gamma\in G}c(Y_s^n,\gamma)
  \big(g(\gamma Y_s^n)-g(Y_s^n)\big) \\
  &\phantom{xx}- g(Y_s^n)\sum_{\gamma\in G}c(Y_s^n,\gamma)
  \big(f(\gamma Y_s^n)-f(Y_s^n)\big)\bigg)\dd s \\
  &= \int_0^t\sum_{\gamma\in G}c(Y_s^n,\gamma)\big(f(\gamma Y_s^n)
  - f(Y_s^n)\big)\big(g(\gamma Y_s^n) - g(Y_s^n)\big)
  = A_{f,g}^n(t).
\end{align*}
Observe that the integrand of $\langle M_f^n,M_g^n\rangle_t$ corresponds, up to the factor 1/2, to the Carr\'e-du-Champ operator associated to the generator $\L_h$. Since each jump has the height $h=2^{-n}\xi$, we have
$$
  \lim_{n\to\infty}\E\Big[\sup_{t\le T}|Y_t^n-Y_{t-}^n|^2\Big]
  \le \lim_{n\to\infty}\E[2^{-2n}\xi^2]=0.
$$
The processes $(B_f^n)$ and $(A_{f,g}^n)$ are continuous in time, which implies for all $T>0$ and $f,g\in\mathcal{A}$ that
$$
  \lim_{n\to\infty}\E\Big[\sup_{t\le T}|B_f^n(t)-B_f^n(t-)|^2\Big]
  = \lim_{n\to\infty}\E\Big[\sup_{t\le T}
  |A_{f,g}^n(t)-A_{f,g}^n(t-)|^2\Big] = 0.
$$
This shows that the processes $M_f^n$ are uniformly integrable and that the limit is a time-continuous process.

Next, we claim that
\begin{equation}\label{a.convp}
  \lim_{n\to\infty}\Pb\bigg(\sup_{t\le T}\bigg|B_f^n(t)
  - \int_0^t\big(\sigma^2\Delta f(Y_s^n)-\na V(Y_s^n)\cdot\na f(Y_s^n)
  \big)\dd s - f(Y_0^n)\bigg|\ge\eps\bigg) = 0.
\end{equation}
To this end, we introduce the differences
$$
  \pa_j^+ f(x)=n(f(x+n^{-1}e_j)-f(x)), \quad
  \pa_j^- f(x)=n(f(x)-f(x-n^{-1}e_j)). 
$$
Then, by adding and subtracting some terms involving $\pa_j V(Y_s^n):=(\pa V/\pa x_j)(Y_s^n)$,
\begin{align*}
  \bigg|B_f^n(t) &- \int_0^t\big(\sigma^2\Delta f(Y_s^n)
  -\na V(Y_s^n)\cdot\na f(Y_s^n)\big)\dd s - f(Y_0^n)\bigg| \\
  &= \bigg|\int_0^t \big[\L_n(f(Y_s^n))
  - \big(\sigma^2\Delta f(Y_s^n)
  -\na V(Y_s^n)\cdot\na f(Y_s^n)\big)\big]\dd s\bigg| \\
  &\le I^+ + I^- + J^+ + J^- + J^0,
\end{align*}
where, inserting the definition of $\L_h$ and $c(i,\pm_j)$,
\begin{align*}
  I^\pm &= \sigma^2 n^2\bigg|\int_0^d\sum_{j=1}^d\bigg[
  \exp\bigg(\mp\frac{\pa_j^\pm V^n(Y_s^n)}{2\sigma^2 n}\bigg)
  - \exp\bigg(-\frac{\pa_j V(Y_s^n)}{2\sigma^2 n}\bigg)\bigg] \\
  &\phantom{xx}\times\big(f(Y_s^n\pm n^{-1}e_j)-f(Y_s^n)\big)\dd s\bigg|, \\
  J^\pm &= \sigma^2 n^2\bigg|\int_0^d\sum_{j=1}^d\bigg[
   \exp\bigg(\mp\frac{\pa_j V(Y_s^n)}{2\sigma^2 n}\bigg)
  - \bigg(1\mp\frac{\pa_j V(Y_s^n)}{2\sigma^2 n}\bigg)\bigg] \\
  &\phantom{xx}\times\big(f(Y_s^n\pm n^{-1}e_j)-f(Y_s^n)\big)\dd s\bigg|,
  \\
  J^0 &= \bigg|\int_0^t\bigg\{\sum_{j=1}^d\bigg[
  \bigg(\sigma^2n^2 - \frac{n}{2}\pa_j V(Y_s^n)\bigg)
  \big(f(Y_s^n+n^{-1}e_j)-f(Y_s^n)\big) \\
  &\phantom{xx}+ \bigg(\sigma^2n^2 + \frac{n}{2}\pa_j V(Y_s^n)\bigg)
  \big(f(Y_s^n-n^{-1}e_j)-f(Y_s^n)\big)\bigg] \\
  &\phantom{xx}+ \na V(Y_s^n)\cdot\na f(Y_s^n) - \sigma^2\Delta f(Y_s^n)
  \bigg\}\dd s.
\end{align*}
The terms $I^\pm$ and $J^\pm$ converge in expectation to zero as $n\to\infty$, since $D$ is  bounded domain, $V$ has a Lipschitz continuous gradient, and $f\in\mathcal{A}$ is smooth. We claim that the expectation of $J^0$ also converges to zero. Indeed, let the discrete Laplacian and gradient be given by
\begin{align*}
  \Delta_n f(i) &= n^2\sum_{j=1}^d
  \big(f(i+n^{-1}e_j)-2f(i)+f(i-n^{-1}e_j)\big), \\
  \na_n f(i) &= \frac{n}{2}
  \big(f(i+n^{-1}e_j)-f(i-n^{-1}e_j)\big)_{j=1}^d,\quad i\in\K_h.
\end{align*}
This gives 
\begin{align*}
  J^0 &\le \sigma^2\bigg|\int_0^t(\Delta_n f(Y_s^n)-\Delta f(Y_s^n))\dd s
  \bigg| \\
  &\phantom{xx}+ \bigg|\int_0^t\na V(Y_s^n)\cdot\big(\na_n f(Y_s^n)
  - \na f(Y_s^n)\big)\dd s\bigg|\to 0.
\end{align*}
It follows from \eqref{a.convp} and the Markov inequality that
\begin{align*}
  \lim_{n\to\infty}&\Pb\bigg(\sup_{t\le T}\bigg|B_f^n(t)
  - \int_0^t\big(\sigma^2\Delta f(Y_s^n)-\na V(Y_s^n)\cdot\na f(Y_s^n)
  \big)\dd s - f(Y_0^n)\bigg|\ge\eps\bigg) \\
  &\le \frac{1}{\eps}\lim_{n\to\infty}\E(I^++I^++J^++J^++J^0) = 0.
\end{align*}

Similarly, we estimate $A_{f,g}^n$:
\begin{align*}
  \bigg|A_{f,g}^n&(t) - 2\sigma^2\int_0^t\na f(Y_s^n)\cdot\na g(Y_s^n)
  \dd s\bigg| \\
  &= \sigma^2\bigg|\int_0^t\bigg\{\sum_{j=1}^d\bigg[
  \exp\bigg(-\frac{1}{2\sigma^2}\big(V^n(Y_s^n+n^{-1}e_j)-V^n(Y_s^n)
  \big)\bigg)\pa_j^+ f(Y_s^n)\pa_j^+ g(Y_s^n) \\
  &\phantom{xx}
  + \exp\bigg(\frac{1}{2\sigma^2}\big(V^n(Y_s^n-n^{-1}e_j)-V^n(Y_s^n)
  \big)\bigg)\pa_j^- f(Y_s^n)\pa_j^- g(Y_s^n)\bigg] \\
  &\phantom{xx}- 2\na f(Y_s^n)\cdot\na g(Y_s^n)\bigg\}\dd s\bigg| \\
  &\le L^+ + L^- + L^0,
\end{align*}
where 
\begin{align*}
  L^\pm &= \sigma^2\bigg|\int_0^t\sum_{j=1}^d\bigg[
  \exp\bigg(\mp\frac{1}{2\sigma^2}\big(V^n(Y_s^n\pm n^{-1}e_j)-V^n(Y_s^n)
  \big)\bigg)-1\bigg] \\
  &\phantom{xx}\times\pa_j^\pm f(Y_s^n)\pa_j^\pm g(Y_s^n)\dd s\bigg|, \\
  L^0 &= \sigma^2\bigg|\int_0^t\big(\pa_j^+ f(Y_s^n)\pa_j^+ g(Y_s^n)
  + \pa_j^- f(Y_s^n)\pa_j^- g(Y_s^n) 
  - 2\na f(Y_s^n)\cdot\na g(Y_s^n)\big)\dd s\bigg|.
\end{align*}
Similarly as above, the expected values of $L^\pm$ and $L^0$ converge to zero as $n\to\infty$ and an application of Markov's inequality shows that
\begin{align*}
  \lim_{n\to\infty}&\Pb\bigg[\sup_{t\le T} \bigg|A_{f,g}^n(t) 
  - 2\sigma^2\int_0^t\na f(Y_s^n)\cdot\na g(Y_s^n)\dd s\bigg|
  \ge\eps\bigg] \\
  &\le \frac{1}{\eps}\lim_{n\to\infty}\E[L^++L^-+L^0] = 0.
\end{align*}

Summarizing the previous results and taking into account \cite[Theorem 7.4.1]{EtKu86}, we conclude that the sequence of processes $(M_f^n)_{n\in\N}$ and also $(B_f^n)_{n\in\N}$ are relatively compact in the space of  c\`adl\`ag functions $f:\R_+^d\to D$, i.e.\ of right-continuous functions for which the left limit exists. Hence, $(f(Y_t^n))_{t\ge 0}$ is relatively compact for all $f\in\mathcal{A}$. Then, for given $f\in\mathcal{A}$, there exists a subsequence $(n_k)_{k\in\N}$ such that $f(Y_t^{n_k})_{t\ge 0}$ converges to some limit $(X_t^f)_{t\ge 0}$. 

Set $f_j(x)=\sin(\pi x_j/(2K))$ for $j=1,\ldots,d$; then $f_j\in\mathcal{A}$. There exists a subsequence $(n_k^1)_{k\in\N}$ of $(n_k)_{k\in\N}$ such that $f_1(Y_t^{n_k^1})$ converges to a limit $X_t^{f_1}$. Since $Y_t^{n_k^1}\in[-K,K]$ is bounded for $t\ge 0$ and $f_j^{-1}(y)=(2K/\pi)\sin^{-1}(y)$ for $y\in[-1,1]$, the processes $Y_t^{n_k^1}$ converges to the limit $\bar{X}_t^1:=f_1^{-1}(X_t^{f_1})$. Repeating this argument for the subsequence $(n_k^{j+1})$ of $(n_k^j)$ for $j=1,\ldots,d-1$, we infer that there exists a subsequence $(\bar{n}_k)_{k\in\N}$ such that for $j=1,\ldots,d$,
$$
  f_j(Y_t^{\bar{n}_k})\mbox{ converges to }X_t^{f_j}, \quad
  Y_t^{\bar{n}_k}\mbox{ converges to }\bar{X}_t^j:=f_j^{-1}(X_t^{f_j}).
$$
Clearly, $(M_f^{\bar{n}_k})_k$, $(B_f^{\bar{n}_k})_k$, and $f(Y_t^{\bar{n}_k})_k$ are relatively compact for $f\in\mathcal{A}$. As in \cite[Theorem 7.4.1]{EtKu86}, $(M_f^{n_k})_k$ is uniformly integrable for $f\in\mathcal{A}$. Consequently,
$$
  M_f(t) = f(\bar{X}_t) - f(\bar{X}_0) - \int_0^t\big(
  \sigma^2\Delta f(\bar{X}_s) - \na V(\bar{X}_s)\cdot\na f(\bar{X}_s)
  \big)\dd s
$$
is a martingale. Now, if $(X_t)_{t\ge 0}$ is a solution to the martingale problem for $\L$ with initial distribution $\mu_0$, the uniqueness of this martingale problem implies that $(Y_t^n)_{t\ge 0}$ converges in distribution to $(X_t)_{t\ge 0}$, finishing the proof.
\end{appendix}


\end{document}